\documentclass[journal,twoside,web]{ieeecolor}

\usepackage{makecell}
\usepackage{generic}
\usepackage{cite}
\usepackage{url}

\usepackage{marvosym}
\usepackage{amsmath}
\usepackage{amssymb}
\usepackage{cases}
\usepackage{amsfonts}
\usepackage{graphicx}
\usepackage{textcomp}
\usepackage{caption}
\usepackage{subcaption}
\usepackage{tabularx}
\usepackage{color}
\let\labelindent\relax
\usepackage{enumitem}
\usepackage[normalem]{ulem}
\usepackage{algorithm,algorithmic}
\usepackage{epstopdf}

\newtheorem{assumption}{Assumption}
\newtheorem{definition}{Definition}
\newtheorem{proposition}{Proposition}
\newtheorem{remark}{Remark}
\newtheorem{theorem}{Theorem}

\DeclareMathOperator*{\argmin}{\arg\min}

\begin{document}

\title{Robustness certificates 
in data-driven non-convex optimization with additively-uncertain  
constraints\thanks{Paper partly supported by FAIR (Future Artificial Intelligence Research) project, funded by the NextGenerationEU program within the PNRR-PE-AI scheme (M4C2, Investment 1.3, Line on Artificial Intelligence),  by the  PRIN PNRR project  P2022NB77E ``A data-driven cooperative framework for the management of distributed energy and water resources'' (CUP: D53D23016100001), funded by the NextGeneration EU program (Mission 4, Component 2, Investment 1.1), by the PRIN 2022 project ``The Scenario Approach for Control and Non-Convex Design'' (project number D53D23001450006), and by the Italian Ministry of Enterprises and Made in Italy in the framework of the project 4DDS (4D Drone Swarms) under grant no. F/310097/01-04/X56.
}
}

\author{Alexander J. Gallo~\IEEEmembership{Member,~IEEE}, Massimiliano Zoggia, Alessandro Falsone~\IEEEmembership{Member,~IEEE},\\ Maria Prandini~\IEEEmembership{Fellow,~IEEE}, and Simone Garatti~\IEEEmembership{Member,~IEEE}
\thanks{The authors are with the Dipartimento di Elettronica Informazione e Bioingegneria, Politecnico di Milano, Piazza Leonardo da Vinci 32, 20133 Milano, Italy (e-mail:  {alexanderjulian.gallo, alessandro.falsone, maria.prandini, simone.garatti}@polimi.it; massimiliano.zoggia@mail.polimi.it).}
}

\maketitle

\begin{abstract}
We consider decision-making problems that are formulated as non-convex optimization programs where uncertainty enters the constraints through an additive term, independent of the decision variables, and robustness is imposed using a finite data-set, according to the scenario robust optimization paradigm.
By exploiting the structure of the constraints, we show that both \textit{a priori} and \textit{a posteriori} distribution-free probabilistic robustness certificates for a possibly sub-optimal solution to the resulting data-driven optimization problem can be obtained with minimal computational effort. Building on these results, we also discuss a one-shot and an incremental procedure to determine the size of the data-set so as to guarantee a user-chosen robustness level. Notably, both the \textit{a posteriori} robustness assessment and incremental data-set sizing do not require to solve the non-convex scenario program. 
A comparative analysis performed on the unit commitment problem using real data reveals a limited increase in conservativeness with a significant computational saving with respect to the application of scenario theory results for general, non necessarily structured, non-convex problems.  
\end{abstract}

\section{Introduction}

In this paper we consider design problems where decisions are taken in the presence of uncertainty that is known from data. We focus on the case when the decision-making problem is formulated as a data-driven optimization program of the following  form
\begin{align} \label{eq:decision-making}
   \min_{x\in\mathcal X} \quad &f(x)\\
   \text{s.t.} \quad &h(x,\delta_i)\leq 0, \quad i= 1, \dots, N, \nonumber
\end{align}
where $x$ is the decision variable and $(\delta_1,\delta_2,\dots,\delta_N)$ is the data-set representing available instances of the  uncertainty $\delta$ taking values in some set $\Delta$. Set $\mathcal X\subseteq \mathbb R^d$ may encode integer constraints, e.g., when it is a subset of 
$\mathbb R^{n_c} \times \{0,1\}^{n_d}$ with $n_c+n_d=d$.  The scalar-valued function $f(\cdot): \mathcal X \to \mathbb R$ is the cost to be minimized, while the $q$-dimensional function $h(\cdot,\cdot):\mathcal X \times \Delta \to \mathbb R^q$ depends on both $x$ and the uncertainty $\delta$ and defines the constraints of the problem. In \eqref{eq:decision-making}, constraints are imposed for all available data to try to enforce robustness against uncertainty. The solution of the constrained optimization program \eqref{eq:decision-making} depends on the data-set $(\delta_1,\dots,\delta_N)$ and is denoted with the shorthand $x^\star_N$ for ease of reference.

The scenario approach \cite{campi2018introduction,CGP:09,campi2021scenario} addresses the issue of evaluating the \textit{robustness} of the decision $x_N^\star$ by assessing the \textit{risk}  that  it violates unseen instances (``scenarios'') of the uncertainty $\delta$, which is modeled as a random outcome from some underlying (typically unknown) probability space $(\Delta,\mathcal D,\mathbb P)$. 

Theoretical  
\cite{campi2008exact,esfahani2014performance,MarPraLyl:14,alamo2015randomized,campi2018wait,9750913,FalMarZizPraGar2022,garatti2024non-convex} as well as  application-oriented \cite{bolognani2016fast,fele2019scenario,FALSONE2019108537,RocCreKen2020,yang2024scenario} results have been developed within the scenario framework. In the earliest theoretical findings in the context of convex robust optimization \cite{campi2008exact}, risk is assessed \textit{a priori}, in terms of an upper bound guaranteed with a certain confidence and associated with the number $N$ of scenarios. More recently, it was shown that tighter risk assessments can be obtained by means of \textit{a posteriori} computed upper bounds  \cite{campi2018wait,garatti2022risk}, depending on the optimal solution $x_N^\star$ and the collected data. This theory applies also to  
non-convex optimization problems and various other decision schemes \cite{garatti2024non-convex,campi2023compression}, for which equivalent \textit{a priori} bounds either do not exist (for general problems) or are very conservative (for specific problems, e.g., \cite{falsone2020l4dc,manieri2023arc}).

In this paper, we consider scenario 
optimization programs of the form 
\begin{align}
    \label{eq:SP}
        \min_{x\in\mathcal X} \quad &f(x)    \\
        \text{s.t.} \quad & g(x) \leq b(\delta_i), \quad i = 1,\dots,N,    \nonumber
    \end{align}
where the constraint function has the following  structure:  $h(x,\delta)= g(x) - b(\delta)$, with $g: \mathcal X \rightarrow \mathbb R^q$ and $b: \Delta \rightarrow \mathbb R^q$ depending only on the decision variable $x$ and the uncertainty $\delta$, respectively. This kind of problems appears in a variety of decision making settings, including model predictive control for linear systems with additive disturbances \cite{schildbach2013scenario,mesbah2016stochastic}, possibly in the presence of quantized inputs, as well as problems in energy management, such as unit commitment \cite{hong2021uncertainty} and economic dispatch \cite{modarresi2019scenario}.
 
In this paper, we explore how the separable structure of the constraints can be exploited to provide the decision maker with a set of \textit{tools} for risk assessment that is wider than those that are available for general, not necessarily structured non-convex problems. In particular,  easy-to-compute robustness certificates
can be provided for any possibly sub-optimal solution of the scenario 
program, and not only for the optimal one. As such, they are relevant in those (non-convex) settings where the solution is hard to compute, e.g.,  because of the presence of local minima or the combinatorial nature of the problem. Specifically, we show that 
\begin{enumerate}[label=\roman*.]
    \item the risk of any feasible solution to \eqref{eq:SP} is \textit{upper bounded} by the risk of an associated, convex, scenario problem;
    \item \textit{a posteriori} robustness certificates on any possibly sub-optimal  scenario 
    solution can be computed with minimal computational effort, without solving the scenario optimization problem; 
    \item \textit{a priori} risk assessment can be performed and used for data-set sizing;
    \item the number of scenarios needed to compute a possibly sub-optimal scenario solution with a certain \textit{a priori} guaranteed robustness level can be reduced using an iterative procedure where the data-set size is grown progressively and the optimization problem is solved only at the final iteration.    
\end{enumerate}
Our results are demonstrated on the unit commitment problem, which consists in determining which generators in a pool should be turned on/off and how much power they should produce over a given time horizon to meet the electricity demand at the lowest cost. Using real-world data, our results show a limited increase in conservativeness, while providing substantial computational advantages in evaluating robustness guarantees, compared to standard scenario theory.

Some results on improved robustness guarantees on the scenario 
solution that exploit the structure of the constraints in the scenario program are also available in \cite{schildbach2014scenario,zhang2015sample}. However, they are confined to the convex case and they provide only \textit{a priori} robustness guarantees for one-shot sizing of the data-set based on  \cite{campi2008exact}.  

A preliminary version of this work was presented in \cite{gallo2025robust}. This paper differs from \cite{gallo2025robust} in that here we consider the risk of any feasible solution to the non-convex scenario program and make explicit that it is upper bounded by the risk of a convex program, so that tighter results from \cite{campi2008exact} can be leveraged. Furthermore, while in \cite{gallo2025robust} we presented a simple artificial numerical example, here we show the effectiveness of our method when applied to the more complex unit commitment problem using \textit{real data} of the hourly aggregate load demand  
from the Spanish transmission system operator (TSO) Red El\'ectrica.

The remainder of this paper has the following structure.
In Section~\ref{sec:recap}, we provide a brief summary of the fundamental concepts of the scenario approach that are relevant to the present work. Reference is made to the data-driven optimization problem \eqref{eq:decision-making} to then show how these general results can be exploited in the structured case \eqref{eq:SP}.  
Specifically, in Section~\ref{sec:main} we show  that the risk of constraint violation associated to any feasible solution to \eqref{eq:SP} is upper bounded by that of the optimal solution of a simpler, convex scenario problem. 
This result is then exploited to provide easily-computable distribution-free probabilistic robustness certificates. Data-set sizing to achieve solutions with a user-defined robustness level  is discussed in Section~\ref{sec:SampleSize}.
Finally, the effectiveness of the proposed approach is demonstrated on the unit commitment problem using real data in Section~\ref{sec:UCP}.
Some concluding remarks are given in Section~\ref{sec:Concl}.

\section{Relevant results from scenario theory}\label{sec:recap}

Consider a scenario program of the form \eqref{eq:decision-making} where $(\delta_1,\dots,\delta_N)$ is a data-set of $N$ uncertainty realizations called ``scenarios'' and modeled as independent draws  from $(\Delta, \mathcal D, \mathbb P)$. The following assumptions are in order. 
\begin{assumption}\label{asm:E!}
    The scenario problem \eqref{eq:decision-making} is feasible with probability one, for every $N$.
    $\hfill\triangleleft$
\end{assumption}
\begin{assumption}\label{asm:uniqueness}
    The scenario problem \eqref{eq:decision-making} admits a unique optimal solution $x_N^\star$ with probability one, for every $N$.    
    $\hfill\triangleleft$
\end{assumption}

With reference to \eqref{eq:decision-making}, a decision $x \in \mathcal X$ is said to be \textit{inappropriate} for the uncertainty realization $\delta \in \Delta$ if the constraint in \eqref{eq:decision-making} is not satisfied for that $\delta$. 
We can then define the \textit{risk} associated to  a decision $x \in \mathcal X$ as  the violation probability
    \begin{equation}\label{eq:rsk}
        V(x) = \mathbb P \{ \delta \in \Delta : \,   h(x,\delta) \not\le 0 \},
    \end{equation}
which is  the probability measure of the set of uncertainty realizations for which $x$ is inappropriate. 

The risk $V(x_N^\star)$ associated with the optimal solution $x_N^\star$ to \eqref{eq:decision-making} is a random quantity since it depends on $x_N^\star$, which is a random variable defined on the product probability space $(\Delta^N, \mathcal D^{\otimes N}, \mathbb P^N)$. The scenario theory allows to determine a bound on $V(x_N^\star)$ that holds with a certain confidence. Such a bound is distribution-free, in that it is valid irrespectively of the underlying probability $\mathbb P$, and it depends on an \textit{observable} quantity called \textit{complexity}, \cite{garatti2024non-convex}. Instrumental to the definition of complexity is that of \textit{support list}.

\begin{definition}[Support list]\label{def:suppSet}
    Given a list of scenarios $(\delta_1,\dots,\delta_N)$, a support list of the scenario program \eqref{eq:decision-making} is any sub-list of length $k \in \{0,1,\dots,N\}$, $\mathcal S_N^k = (\delta_{i_1},\dots,\delta_{i_k})$, with $i_1 < i_2 < \cdots < i_k$ such that:
    \begin{enumerate}[label=\roman*.]
        \item $x_{\mathcal S_N^k}^\star = x_N^\star$, where $x_{\mathcal S_N^k}^\star$ is the minimizer of \eqref{eq:decision-making} with scenarios in $\mathcal S_N^k$ only.
        \item $\mathcal S_N^k$ is irreducible, i.e., no element can be removed from $\mathcal{S}_N^k$ without changing the solution to \eqref{eq:decision-making}.
        $\hfill\triangleleft$
    \end{enumerate}
\end{definition}
\begin{definition}[Complexity]\label{def:cplx}
    For any given data-set $(\delta_1,\ldots,\delta_N)$, the  complexity $s_N^\star$ of the associated scenario program \eqref{eq:decision-making} is the size of the smallest support list. \footnote{The complexity $s_N^\star$ is a function of the scenarios $(\delta_1,\ldots,\delta_N)$ but this dependence is not made explicit for ease of notation.}  
     $\hfill\triangleleft$
\end{definition}

A fundamental result in \cite{garatti2024non-convex} is given by the following theorem, rephrased according to our notation. 

\begin{theorem}[Theorem 6 in \cite{garatti2024non-convex}]\label{th:scenario}
Under Assumptions~\ref{asm:E!} and~\ref{asm:uniqueness}, for any confidence level $\beta \in (0,1)$, the risk $V(x_N^\star)$ of the optimal solution to \eqref{eq:decision-making} satisfies 
\begin{equation} \label{eq:eps-1-beta}
    \mathbb P^N \{V(x_N^\star) \leq \epsilon_{N,\beta}(s_N^\star) \} \geq 1-\beta,
\end{equation}
where the violation function $\epsilon_{N,\beta}: \{0,1,\dots,N\} \rightarrow [0,1]$ is defined as    
\begin{align}\label{eq:eps-def}
    \epsilon_{N,\beta}(k) = \begin{cases}
1-t(k), &k = 0,1,\dots,N-1\\
1, &k=N
\end{cases}
\end{align}
with 
$t(k)$ being the unique solution in $[0,1]$  to
\begin{equation}\label{eq:epsNBeta}
    \frac{\beta}{N} \sum_{m = k}^{N-1} \binom{m}{k} t^{m-k} - \binom{N}{k}t^{N-k} = 0. 
\end{equation}
 $\hfill\Box$
\end{theorem}
Equation \eqref{eq:epsNBeta} can be efficiently solved by bisection; consequently, evaluating $\epsilon_{N,\beta}$ for a given $k$ is computationally cheap (see Appendix B.1 in \cite{campi2023compression} for a ready-to-use MATLAB code). Note that we include both $N$ and $\beta$ in the notation of the function $\epsilon_{N,\beta}(\cdot)$ to highlight that its definition depends explicitly on the number of scenarios in \eqref{eq:decision-making}, and the desired confidence.

Observe that the probability appearing in  \eqref{eq:eps-1-beta} is $\mathbb P^N$ since the random quantities $x_N^\star$ and $s_N^\star$ that are involved depend on the $N$ i.i.d. scenarios $(\delta_1, \dots, \delta_N)$ and are hence defined on the product probability space $(\Delta^N, \mathcal D^{\otimes N}, \mathbb P^N)$. Inequality \eqref{eq:eps-1-beta} means that there is a probability of at least $1-\beta$ that the data-set drawn from $\Delta^N$ is such that its associated complexity $s_N^\star$ leads to a value of $\epsilon_{N,\beta}(s_N^\star)$ for which the bound $V(x_N^\star) \leq \epsilon_{N,\beta}(s_N^\star)$ holds.
Thus, $\epsilon_{N,\beta}(s_N^\star)$ is the robustness level of the optimal solution to the scenario program \eqref{eq:decision-making}, which is guaranteed with confidence $1-\beta$. Function $\epsilon_{N,\beta}(\cdot)$ is non-decreasing so that (as expected) the higher the complexity is, the weaker the robustness guarantees are, in terms of bound on the risk. 

A few remarks are in order to highlight some challenges in applying these results to a non-convex settings. Firstly, they refer to the optimal solution, which may be hard to compute. Secondly, computing the minimal support list may be a very hard combinatorial problem. While useful guarantees can still be obtained by identifying any support list (not necessarily minimal, because function $\epsilon_{N,\beta}(\cdot)$ is non-decreasing), this can remain computationally challenging. Indeed, while in the convex case it is often the case that a support list is almost surely determined by those constraints that are active at the solution (i.e., it is formed by the $\delta_i$'s  such that $h_\ell(x_N^\star,\delta_i)=0$, for some $\ell \in \{1,\dots,q\}$, where $h_\ell$ denotes the $\ell$-th component of the constraint function $h$ in \eqref{eq:decision-making}), for non-convex scenario problems scenarios resulting in non-active constraints may be part of support lists. Thus, finding a support list may require removing scenarios one by one and repeatedly solving the non-convex program. 
Finally, there is in general no \textit{a priori} bound on $s_N^\star$, which may be as large as $N$. This prevents one to \textit{a priori} determine a sufficiently large data-set size guaranteeing that the resulting risk $V(x_N^\star)$ does not exceed a user-chosen threshold with high confidence.

\section{Robustness certificates for non-convex scenario problems with additively-uncertain constraints}\label{sec:main}

This section presents the theoretical contributions of the paper.
We start by showing that the risk of the non-convex scenario problem in \eqref{eq:SP} can be upper bounded by the risk of a suitably defined convex scenario problem, for any data-set of scenarios, and any probability space $(\Delta,\mathcal D,\mathbb P)$.
Based on this result, we derive \textit{a posteriori} and \textit{a priori} risk certificates for problem \eqref{eq:SP}. 

Suppose that the additively-uncertain scenario program \eqref{eq:SP} satisfies the feasibility Assumption~\ref{asm:E!} and let $\hat x_N^\star \in \mathcal X$ denote a possibly sub-optimal yet feasible solution (e.g., obtained by a numerical solver that ensures feasibility but not optimality). 

Consider the following \textit{convex} scenario program
    \begin{align}\label{eq:SP:c}
             \max_{\xi \in \mathbb R^q} \hspace{.5cm}& 
             \sum_{\ell = 1}^q \xi_\ell \\
             \text{s.t.} \hspace{.5cm}
             & 
             \xi_\ell \leq b_\ell(\delta_i), \quad \ell=1,\dots,q, \; i = 1,\dots,N, \nonumber
    \end{align}
where  $b_\ell(\cdot)$ is the $\ell$-th component of the function $b(\cdot)$ appearing in the constraints of \eqref{eq:SP}.   
The linear program \eqref{eq:SP:c} is always feasible and has a unique solution given by $\xi_N^\star=[\xi_{N,1}^\star\,\ldots\, \xi_{N,q}^\star]^\top$, where
\begin{equation}\label{eq:solXi}
        \xi_{N,\ell}^\star = \min_{i=1,\dots,N} b_\ell(\delta_i), \quad \ell=1,\dots,q.
    \end{equation}
Denote with $V'(\xi_N^\star)$ the risk associated to the solution $\xi_N^\star$ of the scenario program \eqref{eq:SP:c}: 
    \[
    V'(\xi_N^\star) = \mathbb P\{\delta \in \Delta : \xi_{N}^\star \not\le b(\delta)\}.
    \]
We can now state our first result. 

\begin{theorem}\label{th:comp}
Consider the scenario programs \eqref{eq:SP} and \eqref{eq:SP:c} associated to the same data-set $(\delta_1, \dots, \delta_N)$ of scenarios drawn independently from $(\Delta, \mathcal D, \mathbb P)$. For any probability $\mathbb P$, for any data-set $(\delta_1, \dots, \delta_N)$, the risk of a possibly sub-optimal solution $\hat x_N^\star$ of the scenario program \eqref{eq:SP} is upper bounded by the risk of the optimal solution $\xi_N^\star$ of the scenario program \eqref{eq:SP:c}, i.e., 
    \begin{equation}\label{eq:boundRisk}
        V(\hat x_N^\star) \leq V'(\xi_N^\star).
    \end{equation}
\end{theorem}
\begin{proof}
For each $\ell = 1,\dots, q$, the constraints $g_\ell(x) \leq b_\ell(\delta_i)$, $i = 1,\dots,N$, in \eqref{eq:SP} are nested and dominated by the one with the smallest $b_\ell(\delta_i)$. Thus, the optimal solution $\xi_{N}^\star$  to the linear program \eqref{eq:SP:c} given in \eqref{eq:solXi} defines the whole feasibility region of \eqref{eq:SP}, by using $g(x) \leq \xi_N^\star$ in place of its constraints. It therefore follows that $g_\ell(\hat x_N^\star) \leq \xi_{N,\ell}^\star$ is satisfied for all $\ell = 1,\dots, q$. Because of this, for any $\delta \in \Delta$, the following implication holds
    \begin{equation*}
        g_\ell(\hat x_N^\star) > b_\ell(\delta) \implies \xi_{N,\ell}^\star > b_\ell(\delta).
    \end{equation*}
    Thus, for any $\ell = 1,\dots, q$, 
    \begin{equation*}
 \{\delta\in\Delta:  g_\ell(x_N^\star)>b_\ell(\delta)\}
        \subseteq 
        \{\delta\in\Delta: \xi_{N,\ell}^\star>b_\ell(\delta)\},
    \end{equation*}
    which implies
    \begin{align*}
        &V(\hat x_N^\star) = \mathbb P\{\delta\in\Delta: \exists \ell \in \{1,\dots,q\} \text{ s.t. } g_\ell(\hat x_N^\star)>b_\ell(\delta)\}\\
        &\le V'(\xi_N^\star) = \mathbb P\{\delta \in \Delta : \exists \ell \in \{1,\dots,q\} \text{ s.t. } \xi_{N,\ell}^\star >  b_\ell(\delta)\},
    \end{align*}
    thus concluding the proof since all derivations hold for any data-set $(\delta_1, \dots, \delta_N)$ and any~$\mathbb P$.    
\end{proof}

Theorem~\ref{th:comp} exposes a simple but fundamental property of the non-convex scenario program in \eqref{eq:SP}. 
Indeed, it demonstrates that the risk associated to any feasible, possibly sub-optimal, solution $\hat x_N^\star$ is related to the risk of the solution of a much simpler, convex linear program, $\xi_N^\star$.
The two are not, however, equivalent, as there may exist $\delta \in \Delta$ such that $g_\ell(\hat x_N^\star) \leq b_\ell(\delta)$ holds for all $\ell = 1,\dots, q$, while there is an $\ell$ for which $b_\ell(\delta) < \xi_{N,\ell}^\star$ (this corresponds to the situation in which adding a $\delta$ and the corresponding constraints reduces the feasible region of \eqref{eq:SP} while maintaining $\hat x_N^\star$ within it).  
A representation of this gap is presented in Figure~\ref{fig:gVxi}, for a simplified case in which $q = 1$.

\begin{figure}
    \centering
    \includegraphics[width=0.95\linewidth]{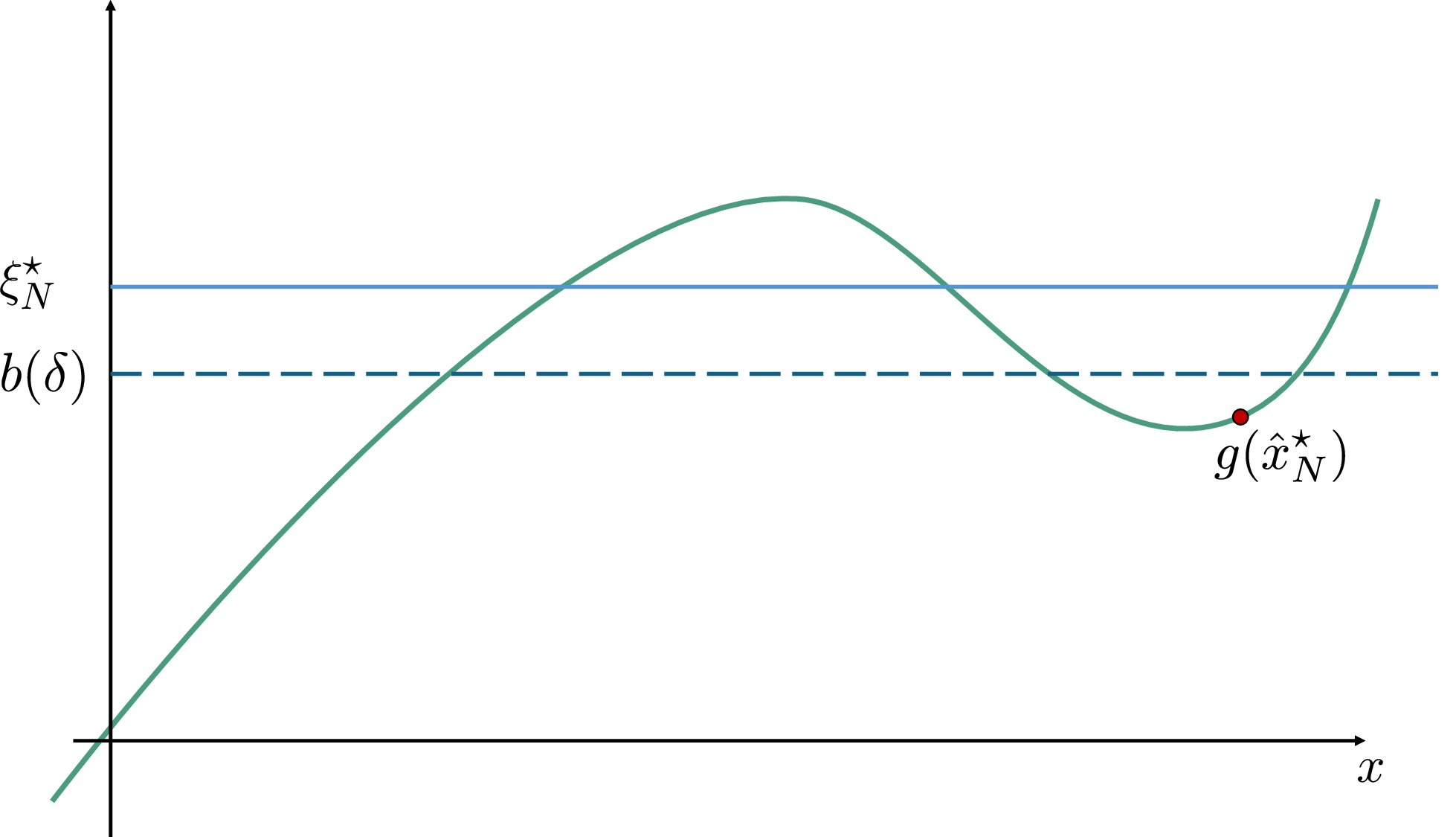}
    \caption{Visualization of the gap between $V(\hat x_N^\star)$ and $V'(\xi_N^\star)$ when $q=1$. The continuous horizontal line represents $\xi_{N}^\star$, defining the feasible region $g(x) \leq \xi_{N}^\star$; the solution $\hat x_N^\star$ is such that $g(\hat x_N^\star) < \xi_{N}^\star$ (red dot). 
    The dashed horizontal line is the value of $b(\delta)$ for some $\delta \in \Delta$. While ${b(\delta) < \xi_{N}^\star}$, it holds that  ${g(\hat x_N^\star) < b(\delta)}$, meaning that $x_N^\star$ is appropriate for this value of $\delta \in \Delta$.}
    \label{fig:gVxi}
\end{figure}

Even if the two problems do not have equivalent risk, bounding the risk of the non-convex scenario problem with that of a simpler convex one may present  significant benefits, as demonstrated throughout the remainder of the paper.

We next present the consequences of Theorem~\ref{th:comp} in terms of robustness certificates. 

\begin{proposition}[a posteriori robustness certificate]\label{prop:a-posteriori-robustness}
Consider the non-convex additively-uncertain  scenario problem \eqref{eq:SP}. For any $\mathbb{P}$,  the risk of its possibly sub-optimal solution $\hat x_N^\star$ satisfies:
    \begin{equation}\label{eq:a-posteriori-bound}
        \mathbb P^N \{V(\hat x_N^\star) \leq \epsilon_{N,\beta}(\varsigma_N) \} \geq 1 - \beta,
    \end{equation}
where $\epsilon_{N,\beta}(\cdot)$ is defined in \eqref{eq:eps-def} and  $\varsigma_N$ is the number of elements in the set $\mathcal I_N = \{i_1,\dots,i_q\}$, i.e., the number of distinct $i_\ell$'s, with $i_\ell$ given by\footnote{\label{ftnt:tie}In \eqref{eq:idx}, if the minimum is attained for two (or more) scenarios $\delta_i, \delta_j$, we then assume a unique $i_\ell$ is selected via a tie-break rule, e.g., $i_\ell = \min \{i,j\}$.}
    \begin{equation}\label{eq:idx}
        i_\ell = \argmin_{i = 1,\dots,N} \, b_\ell(\delta_i)
    \end{equation}
for each $\ell=1,\dots,q$. 
\end{proposition}
\begin{proof}
From Theorem~\ref{th:comp}, it follows that 
\begin{equation*}
        \mathbb P^N \{V(\hat x_N^\star) \leq \epsilon_{N,\beta}(\varsigma_N) \} \geq         \mathbb P^N \{V'(\xi_N^\star) \leq \epsilon_{N,\beta}(\varsigma_N) \} 
\end{equation*}
so that we just need to prove that  
\begin{equation}\label{eq:bound}
        \mathbb P^N \{V'(\xi_N^\star) \leq \epsilon_{N,\beta}(\varsigma_N) \} \geq 1 - \beta.
\end{equation}
To this aim, we start by noting that, for each $\ell = 1,\dots, q$, $\xi_\ell \leq b_\ell(\delta_i)$, $i = 1,\dots,N$, defines a set of nested constraints that are dominated by the one with $b_\ell(\delta_{i_\ell})$. It follows that $\xi_{\mathcal{I}_N}^\star = \xi_N^\star$, where $\xi_{\mathcal{I}_N}^\star$ is the solution of \eqref{eq:SP:c} with only those scenarios indexed by $\mathcal I_N$ in place. 
Thus, the list of scenarios indexed by $\mathcal I_N$ surely contains a support list for \eqref{eq:SP:c} as a sub-list,\footnote{In fact, $\mathcal I_N$ yields the minimal support list whenever, for each $\ell$, the minimum $b_\ell(\delta_i)$ in \eqref{eq:idx} is attained by a unique scenario. When, instead, a tie-break rule as in footnote \ref{ftnt:tie} must be invoked, $\mathcal I_N$ may even fail to be irreducible. To always attain the minimal support list, a tie-break rule minimizing $\varsigma_N$ can be adopted, albeit at the expense of higher computational effort.}  and we have that the complexity $s_N^\star$ of problem \eqref{eq:SP:c} is upper bounded by  $\varsigma_N$. Since $\epsilon_{N,\beta}(\cdot)$ is non-decreasing for any $N$ and $\beta$, equation \eqref{eq:bound} follows from Theorem~\ref{th:scenario}.
\end{proof}

\begin{remark}[comparison with previous works] \label{rmk:comp}
In \cite{gallo2025robust}, the result in Proposition~\ref{prop:a-posteriori-robustness} was proven for the optimal solution $x^\star_N$ of \eqref{eq:SP} by using \eqref{eq:eps-1-beta} in the scenario Theorem~\ref{th:scenario} and observing that the complexity $s_N^\star$ satisfies 
$s_N^\star \le \varsigma_N \le q$. Indeed, $\mathcal I_N$ defines the feasibility region of \eqref{eq:SP}, so that the solution $x_{\mathcal I_N}^\star$ of \eqref{eq:SP} with only those scenarios indexed by $\mathcal I_N$ in place satisfies $x_{\mathcal I_N}^\star = x_N^\star$. Hence, the list of scenarios indexed by $\mathcal I_N$ must be (or contain) a support list as per Definition~\ref{def:suppSet}. 
 $\hfill\triangleleft$
\end{remark}

Besides being applicable to any feasible suboptimal solution, the main advantage of Proposition~\ref{prop:a-posteriori-robustness} lies in the negligible computational effort required to compute $\varsigma_N$. Consequently, the certification $\epsilon_{N,\beta}(\varsigma_N)$ is inexpensive to obtain. 
	
The downside of Proposition \ref{prop:a-posteriori-robustness} is that $\epsilon_{N,\beta}(\varsigma_N)$ essentially evaluates $V'(\xi_N^\star)$, which, as discussed earlier, may differ from $V(\hat{x}_N^\star)$. This discrepancy may lead to conservative evaluations. Nevertheless, although no formal guarantees can be provided, in many additively-uncertain problems of interest the gap tends to be minor, as also illustrated by the numerical example in Section~\ref{sec:UCP}. Furthermore, using Proposition~\ref{prop:a-posteriori-robustness} for a first evaluation does not preclude the user from taking additional actions 
if the initial, computationally inexpensive certification $\epsilon_{N,\beta}(\varsigma_N)$ proves unsatisfactory.
For instance, when the optimal solution $x_N^\star$ can be computed, one can try to obtain a more accurate evaluation of the actual complexity at the price of additional computational effort, as discussed in Section \ref{subsec:UCP-robustness} with reference to the unit commitment problem.
\begin{remark}[computational considerations] \label{rmk:computational}
The fact that $\xi_N^\star$ suffices to characterize the feasible region of additively-uncertain problems can be exploited to also reduce the computational complexity to obtain an optimal or suboptimal solution $x_N^\star$ or $\hat{x}_N^\star$, by solving (exactly or approximately)
\begin{align} \label{eq:reduced-SP:c}
	\min_{x\in\mathcal X} \quad &f(x)    \\
	\text{s.t.} \quad & g_\ell(x)\leq \xi^\star_{N,\ell}, \quad \ell=1,\dots,q,  \nonumber
\end{align}
in place of \eqref{eq:SP}. This typically entails discarding a large number of constraints, since in general $q \ll N$. $\hfill\triangleleft$
\end{remark}

A priori robustness guarantees can be provided by observing that $\varsigma_N$ is upper bounded by $q$ and recalling that $\epsilon_{N,\beta}(\cdot)$ is non-decreasing for any $N$ and $\beta$, thus getting 
\begin{equation}\label{eq:a-priori-with-a-posteriori} 
    \mathbb P^N \{V(x_N^\star) \leq \epsilon_{N,\beta}(q) \} \geq 1-\beta
\end{equation}
from equation \eqref{eq:a-posteriori-bound} in Proposition~\ref{prop:a-posteriori-robustness}. 
An improved bound can be derived by exploiting Theorem~\ref{th:comp} and the \textit{a priori} guarantees for convex scenario problems in \cite{campi2008exact}. 

\begin{proposition}[a priori robustness certificate]\label{prop:a-priori-robustness}
Consider the non-convex additively-uncertain  scenario problem \eqref{eq:SP}. For any $\mathbb{P}$,  the risk of its possibly sub-optimal solution $\hat x_N^\star$ satisfies:
    \begin{equation}\label{eq:nc:aPri}
        \mathbb P^N \{V(\hat x_N^\star) \leq \varepsilon \} \geq 1 - \sum_{m = 0}^{q-1} \binom{N}{m} \varepsilon^m (1-\varepsilon)^{N-m},
    \end{equation}
    for any $\varepsilon \in (0,1)$, where $q$ is the number of components of function $g(\cdot)$ in \eqref{eq:SP}.
\end{proposition}
\begin{proof}
Based on Theorem~\ref{th:comp}, we have that 
\begin{equation*}
\mathbb P^N \{V(\hat x_N^\star) \leq \varepsilon \} \ge \mathbb P^N \{V'(\xi_N^\star) \leq \varepsilon \}.
\end{equation*}
In turn, given that the scenario program \eqref{eq:SP:c} is convex, it satisfies Assumptions~\ref{asm:E!} and~\ref{asm:uniqueness}, and its feasibility domain has non-empty interior, we can apply \cite[Theorem 1]{campi2008exact}, thus obtaining 
    \begin{equation*}
\mathbb P^N \{V'(\xi_N^\star) \leq \varepsilon \} \geq 1 - \sum_{m = 0}^{q-1} \binom{N}{m} \varepsilon^m (1-\varepsilon)^{N-m}
    \end{equation*}
(note that $q$ is also the dimension of the decision variable $\xi$ in \eqref{eq:SP:c}).
\end{proof}
Proposition~\ref{prop:a-priori-robustness} states that the distribution of the risk is upper bounded by a Beta distribution. Given a confidence parameter $\beta \in (0,1)$, an \textit{a priori} robustness guarantee that holds with confidence $1-\beta$ is obtained by computing the minimum value of $\varepsilon \in (0,1)$ satisfying  
\begin{equation*}
    \sum_{m=0}^{q-1} \binom{N}{m} \varepsilon^m(1-\varepsilon)^{N-m} \leq \beta,
\end{equation*}
so that $\mathbb{P}^N \{V(\hat x_N^\star) \leq \varepsilon \} \geq 1-\beta$.

\begin{remark}[tightness of the \textit{a priori} robustness guarantees]\label{rem:a-priori}
It is worth noting that the result is tighter than the one in \eqref{eq:a-priori-with-a-posteriori} in that the obtained $\varepsilon$ is smaller than  $\epsilon_{N,\beta}(q)$. As noted in the discussion after \cite[Corollary~8]{garatti2024non-convex}, Proposition~\ref{prop:a-priori-robustness} does not generally hold for non-convex problems with complexity upper bounded by $q$. The validity of the result in the present setting relies crucially on the structure of problem~\eqref{eq:SP}. This bears some similarities with some results in the literature of the scenario theory developed for convex problems, \cite{schildbach2014scenario,zhang2015sample}. $\hfill\triangleleft$  
\end{remark}

\section{Experiment sizing in non-convex scenario problems with additively-uncertain constraints}\label{sec:SampleSize}
\noindent

In this section, we discuss how the bounds on the risk derived in Propositions~\ref{prop:a-posteriori-robustness} and~\ref{prop:a-priori-robustness} can be used not only to assess the robustness properties of the solution to the additively-uncertain  scenario program \eqref{eq:SP}, but also to decide the size of the data-set so as to secure desired robustness levels. 
This shows that the results of \cite{campi2008exact} and \cite{garatti2023complexity} also apply in the non-convex additively-uncertain  setting.

\subsection{One-shot data-set sizing using the \textit{a priori} guarantees}

\noindent
Proposition~\ref{prop:a-priori-robustness} with the robustness guarantees in equation \eqref{eq:nc:aPri} immediately provides a means to \textit{a priori} determine the size $N$ of the data-set $(\delta_1, \dots,\delta_N)$ so as to guarantee a certain user-defined bound $\bar\epsilon \in (0,1)$ on the risk level with a given confidence $1-\beta$ as stated in the following theorem. 

\begin{theorem} \label{th:1shot}
 Given $\beta \in (0,1)$ and $\bar\epsilon \in (0,1)$, set the data-set size $N$ as 
    \begin{equation}\label{eq:oneshot}
         N=\min\big\{M \ge q: \, \sum_{m = 0}^{q-1}\binom{M}{m}\bar\epsilon^m (1-\bar\epsilon)^{M-m} \leq  
         \beta\big\}.  
    \end{equation}   
Then, for any probability $\mathbb P$, any possibly sub-optimal solution $\hat x_{N}^\star$ to the scenario problem \eqref{eq:SP} with a data-set of size $N$ satisfies 
    \begin{equation}\label{eq:bound:oneshot}
        \mathbb P^{N} \{V(\hat x_{N}^\star) \leq \bar\epsilon\} \geq 1 - \beta.
    \end{equation} 
\end{theorem}
\begin{proof}
    The result follows directly from Proposition~\ref{prop:a-priori-robustness}, as replacing \eqref{eq:oneshot} in \eqref{eq:nc:aPri} leads to \eqref{eq:bound:oneshot}.
\end{proof}

\begin{figure}
    \centering
    \includegraphics[width=0.95\linewidth]{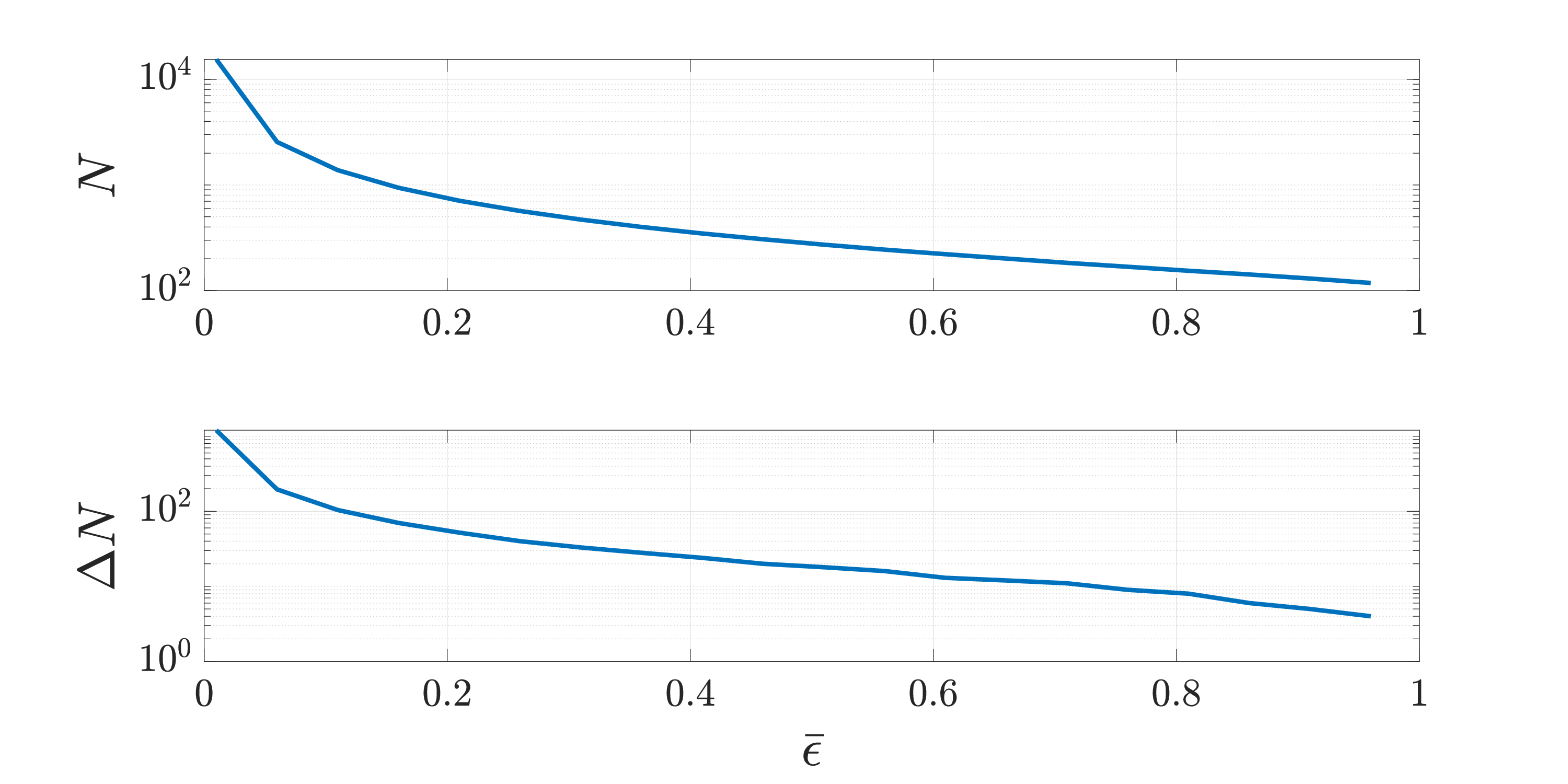}
    \caption{Data-set size $N$ computed with \eqref{eq:oneshot} (top plot) and  increment $\Delta N$ of the data-set size when computed with \eqref{eq:Ncheck}  (bottom plot) as a function of $\bar \epsilon$ for $q = 100$ when $\beta = 10^{-6}$.} 
    \label{fig:Ncomp}
\end{figure}

Figure~\ref{fig:Ncomp} shows the value of $N$  in \eqref{eq:oneshot} and the difference  with respect to the (larger, see Remark~\ref{rem:a-priori}) value of $N$ calculated via 
\begin{equation}\label{eq:Ncheck}
    N = \min\{ M \geq q: \epsilon_{M,\beta}(q)\leq \bar{\epsilon}\},  
\end{equation}
as a function of $\bar\epsilon$, for $q = 100$ with $\beta = 10^{-6}$.

\subsection{Incremental data-set sizing using the \textit{a posteriori} guarantees}

\noindent
While the selection of $N$ as in Theorem \ref{th:1shot} safeguards against the worst possible situations, as highlighted in \cite{garatti2023complexity} with reference to the convex setting, depending on the problem, it may be that for some realizations of the scenarios $\delta_1,\ldots,\delta_N$ one could use only a smaller portion of them, and yet satisfy the bound $\bar \epsilon$ on the actual risk. Reducing $N$ is advantageous for example when  collecting data  is costly. 

To overcome the conservativeness of the result in Theorem~\ref{th:1shot}, we shall next adapt  the incremental approach proposed in \cite[Algorithm 1]{garatti2023complexity} to the non-convex additively-uncertain  scenario problem \eqref{eq:SP}. In this approach, while the data-set sizing is \textit{a posteriori} tuned, an \textit{a priori} certification of the solution robustness level is retained.

Given $\bar\epsilon \in (0,1)$ and $\beta \in (0,1)$, define 
\begin{align}
N_j =  \min \Big\{ &N >\bar{M}_j :  \vphantom{ \sum_{m=j}^{\bar{M}_j} } \label{eq:defNj}\\
            &
            \beta_j \sum_{m=j}^{\bar{M}_j} \binom{m}{j} (1-\bar\epsilon)^{m-j} \geq \binom{N}{j}(1-\bar\epsilon)^{N-j}
            \Big\} \nonumber
\end{align}
with $\beta_j = \tfrac{\beta}{(q+1)(\bar{M}_j+1)}$, for $j = 0,1,\dots,q$, where 
\begin{equation*}
    \bar{M}_j = \min \left\{
    M \geq j : \sum_{m = 0}^{j-1} \binom{M}{m} \bar\epsilon^m (1-\bar\epsilon)^{M-m} \leq \beta
    \right\},
\end{equation*}
for $j = 1,\dots,q$ and $\bar M_0 = \bar M_1$.
The incremental approach is described in Algorithm~\ref{alg:Nj} and it works as follows: at each iteration $j$, $N_j$ scenarios are gathered by adding scenarios to those already collected and $\varsigma_{N_j}$ is computed as an upper bound on the complexity of \eqref{eq:SP:c} with $N_j$ in place of $N$; 
then, if $\varsigma_{N_j} > j$, $j$ is increased and a new iteration is started;
otherwise, problem~\eqref{eq:SP} for the presently collected $N_j$ scenarios is solved (via formulation \eqref{eq:reduced-SP:c} to reduce the computational complexity -- see Remark~\ref{rmk:computational}) and the result $x^\star$ is returned. The algorithm is guaranteed to terminate since $\varsigma_{N_q} \leq q$ is always guaranteed.

\begin{algorithm}[t]
    \caption{Incremental scenario solution}\label{algo}
    \begin{algorithmic}[1]
        \renewcommand{\algorithmicrequire}{\textbf{Input:}}
        \renewcommand{\algorithmicensure}{\textbf{Output:}}
        \REQUIRE $N_j$, $j = 0,1,\dots, q$, from \eqref{eq:defNj}
        \ENSURE $x^\star$ 
        \STATE $j \gets 0$, $N_{-1} \gets 0$;
        \STATE collect i.i.d. data $\delta_{N_{j-1}+1},\dots \delta_{N_j}$, independent of previous scenarios
        \STATE $\mathcal I_{N_j} \gets \{i_{1},\dots, i_q\}$ with $i_\ell=     \arg\min_{i = 1,\dots,N_j} b_\ell(\delta_i)$ 
        \STATE $\varsigma_{N_j} \gets$ number of distinct  indices in   $\mathcal{I}_{N_j}$ \label{step-complexity}
        \IF{$\varsigma_{N_j} \leq j$}
            \STATE {
                 $x^\star \gets$ possibly sub-optimal solution to \\
                 $\begin{array}{rl}
                    \displaystyle \min_{x\in\mathcal X} & f(x) \\
                    \text{s.t.}& g_\ell(x) \leq b_\ell(\delta_{i_\ell}), \; \ell = 1,\dots,q
                \end{array}$
            }
            \STATE \textbf{return $x^\star$}             
        \ELSE
            \STATE $j \gets j+1$;
            \STATE go to step 2;
        \ENDIF
    \end{algorithmic}
    \label{alg:Nj}
\end{algorithm}

The following theorem summarizes the result.
\begin{theorem}\label{th:incr}
Given $\bar\epsilon \in (0,1)$ and $\beta \in (0,1)$, the output  $x^\star$ returned by Algorithm~\ref{alg:Nj} satisfies
    \begin{equation*}
        \mathbb P^{N_q} \{ V(x^\star) \leq \bar \epsilon \} \geq 1-\beta,
    \end{equation*}
for any probability $\mathbb P$.
\end{theorem}
\begin{proof}
Recall that by Theorem~\ref{th:comp}, the risk of any possibly sub-optimal solution to the non-convex scenario program \eqref{eq:SP} is upper bounded by that of the solution to the scenario program \eqref{eq:SP:c} computed with the same data-set. This means that 
\[
\mathbb P^{N_q} \{ V(x^\star) > \bar \epsilon \}
\le \mathbb P^{N_q} \{ V'(\xi^\star) > \bar \epsilon \}
\]
where $\xi^\star$ denotes the solution of \eqref{eq:SP:c} with the constraints corresponding to the $N_j$ scenarios available at termination of Algorithm \ref{algo}. The result then follows directly from \cite[Theorem 1]{garatti2023complexity}, since $\xi^\star$ coincides with the solution returned when \cite[Algorithm 1]{garatti2023complexity} is used with the convex scenario program \eqref{eq:SP:c}.\footnote{Note that \cite[Algorithm~1]{garatti2023complexity} is stated in terms of the  actual complexity computed at each iteration and for so-called non-degenerate problems (in this case, the complexity coincides with the number of support scenarios mentioned in \cite{garatti2023complexity}); however, the results of \cite{garatti2023complexity} remain straightforwardly valid without non-degeneracy and when an upper bound on the complexity is used.}

\end{proof}

\section{Application to the unit commitment problem}\label{sec:UCP}

\subsection{The unit commitment problem as a scenario program}

Unit commitment is a fundamental problem to be addressed for the operation of the electricity grid. It consists in deciding which power generators to turn on or off over time so as to meet the forecasted electricity demand while minimizing  costs and satisfying system constraints,  \cite{padhy2004unit,zheng2014stochastic}. 
In this paper, we adopt the  single bus approximation, thus ignoring the grid model and the related transmission line limits.  We focus on the case when thermal generators (including gas turbines in simple and combined cycle units, cogeneration plants and nuclear reactors) are to be scheduled over a one-day horizon discretized in $T = 24$ time slots of one-hour each.  Thermal generators are characterized by operational constraints in terms of minimum up-time and down-time, 
ramping limitations, and prohibited intervals of power production, \cite{UC-book-2014}. 

Given $n_{\mathrm{p}}$ thermal Generation Units (GUs) and  a data-set of $N$ demand profiles $\delta_i= P^{\mathrm{d}}_{i} = [P^{\mathrm{d}}_{i,0} \, \cdots \, P^{\mathrm{d}}_{i,T-1}]^\top \in \mathbb R^T$, $i=1,\dots,N$, we next formulate the unit commitment problem  as a additively-uncertain  scenario  optimization problem where the fuel and operating costs are minimized subject to the operating constraints of each single generator and a coupling constraints related to the robust-over-the-data satisfaction of the demand. 

Not surprisingly, the resulting problem is non-convex since we are not only optimizing the power production $P_{j,t}$ per time slot $t$ of each GU $j$ but also defining its commitment status  which is naturally modeled via a binary variable $Y_{j,t}$. The fact that  GU $j$ has  prohibited intervals of productions calls for the introduction of additional binary variables. Specifically, suppose that $P_{j,t}$ can take values in a disconnected interval which is the union of $Z_j\ge 1 $ non-overlapping operating zones $ [\underline{P}_{j,z}, \overline{P}_{j,z}]$, $z=1,\dots, Z_j$. 
Then, the fact that $P_{j,t}$ either belongs to one such zone interval or is zero can be modeled through $Z_j$ binary variables $y_{j,z,t}$, each one indicating if GU $j$ is operating in the corresponding zone $ [\underline{P}_{j,z}, \overline{P}_{j,z}]$ in slot $t$, and  whose sum over the operating zone index $z$ provides the commitment status variable $Y_{j,t}$.  The variation of the power production $P_{j,t}$  within a zone interval or between different zone intervals is subject to rate constraints.
Further binary variables $u_{j,t}$ and $d_{j,t}$ are introduced in the problem formulation to account for the minimum up-time $T_j^{\mathrm u}$ and down-time $T_j^{\mathrm d}$, i.e., the minimum number of time slots that the GU $j$ has to stay on/off when it has been switched on/off, that are typically much smaller than $T$. Variable $u_{j,t}$ ($d_{j,t}$) is set to 1 if and only if the GU $j$ is switched on (off) at time slot $t$, and from that $t$ the count starts and the commitment status variable $Y_{j,t}$ has to stay equal to 1 (0) for at least  $T_j^{\mathrm{u}}$ ($T_j^{\mathrm{d}}$) time slots, including $t$.    
The binary variables $Y_{j,t}$, $u_{j,t}$, and $d_{j,t}$ respectively allow also to account for operating costs related to the commitment, activation and deactivation of GU $j$ at time $t$. These costs are added to the fuel cost represented by a quadratic function of the power $P_{j,t}$ provided by the GU $j$ at time $t$. 

The resulting mixed-integer unit commitment scenario program is given by: 
\begin{subequations}\label{eq:UC}
    \allowdisplaybreaks
    \begin{align}
        \min_{ \substack{P_{j,t}, y_{j,z,t} \\ u_{j,t}, d_{j,t} } }
        &\sum_{t = 0}^{T-1} \sum_{j = 1}^{n_{\mathrm{p}}} 
        \left(a_j P_{j,t}^2 + b_j P_{j,t} + c_j Y_{j,t}  +c_j^{\mathrm{u}} u_{j,t} + c_j^{\mathrm{d}} d_{j,t} \right) \nonumber 
        \\
        \,\, \text{s.t.} \,
        &\sum_{j = 1}^{n_{\mathrm{p}}} P_{j,t} \geq P^{\mathrm{d}}_{i,t} \,, \quad i = 1,\dots,N \label{eq:UC:demand}\\
        & 
        -\underline{\Delta}^{\mathrm{p}}_j \leq P_{j,t} - P_{j,(t-1)\,\text{mod}\,T} \leq \overline{\Delta}^{\mathrm{p}}_j \label{eq:UC:ramp}\\
        & \sum_{z = 1}^{Z_j} y_{j,z,t}\underline{P}_{j,z} \leq P_{j,t} \leq \sum_{z = 1}^{Z_j} y_{j,z,t}\overline{P}_{j,z}  \label{eq:UC:maxMinP}\\
        & Y_{j,t} = \sum_{z = 1}^{Z_j} y_{j,z,t} \label{eq:UC:ON/OFF}\\
        & Y_{j,t} \leq 1 \label{eq:UC:onlyOne} \\
        & Y_{j,t} - Y_{j,(t-1)\,\text{mod}\,T} \leq u_{j,t} \leq Y_{j,t} \label{eq:UC:swON}\\
        & u_{j,t}\le 1-Y_{j, (t-1) \,\text{mod}\,T}\label{eq:UC:swON-bis}\\
        & Y_{j,\tau\,\text{mod}\, T} \geq     u_{j,t}, \quad \tau \in [t,t+T^{\mathrm{u}}_j-1] \label{eq:UC:uptime}\\
        & Y_{j,(t-1)\,\text{mod}\,T} - Y_{j,t} \leq d_{j,t} \leq Y_{j,(t-1)\,\text{mod}\,T} \label{eq:UC:swOFF}\\
        & d_{j,t}\le 1-Y_{j,t}\label{eq:UC:swOFF-bis}\\
        & Y_{j,\tau \,\text{mod}\,T} \leq 1 - d_{j,t}, \quad \tau \in [t,t+T^{\mathrm{d}}_j-1] \label{eq:UC:dwtime}\\
        & y_{j,z,t}, u_{j,t}, d_{j,t} \in\{0,1\},\, P_{j,t}\in \mathbb R \\
        & j=1,\dots, n_{\mathrm{p}}, \, \, t = 0,1,\dots,T-1. \nonumber
\end{align}
\end{subequations}

Note that the cost function to be minimized is composed of the fuel cost per time slot of each unit given by the first three terms in the double summation, \cite{ZivicDjurovic2012},  and the startup and shutdown costs. 
In fact, in the startup phase, a thermal unit consumes a fixed amount of fuel to reach the desired temperature and pressure, \cite{KirschenStrbac2004}, whereas it wastes fuel during the shutdown procedure, \cite{BaldwinDaleDittrich1959}. Decision variables and parameters are summarized in Table~\ref{tab:UC:vars}.  

The problem is formulated so as to provide a one-day solution that could be applied over multiple consecutive days, in the presence of a stationary demand profile. This is achieved by imposing that when the time index exceeds the extremes of the interval $[0,T-1]$ it is reset  to an index in the previous day (if smaller than 0) or in the following day (if larger than $T-1$)  such that the distance from the extremes is preserved. This is imposed by evaluating the time index modulo $T$, so that, e.g., $(-1) \,\text{mod}\,T=T-1$ and $(T+k) \,\text{mod}\,T=k$ for $k=0,\dots,T-1$, as shown in the ramp constraint \eqref{eq:UC:ramp}, the switch on conditions \eqref{eq:UC:swON} and \eqref{eq:UC:swON-bis}, and the switch off condition \eqref{eq:UC:swOFF} when evaluated at $t=0$, and in the minimum on-time and off-time constraints  \eqref{eq:UC:uptime} and \eqref{eq:UC:dwtime}, when evaluated towards the end of the one-day time horizon.         

The $q=T$ constraints in \eqref{eq:UC:demand}, which enforce demand satisfaction, are those in which scenarios show up and exhibit an additively-uncertain structure. Hence,  the unit commitment scenario problem \eqref{eq:UC} is of the form \eqref{eq:SP}, with the decision variable $x$ collecting all scalar variables $y_{j,z,t}$, $P_{j,t}$, $u_{j,t}$, $d_{j,t}$, $z=1,\dots, Z_j$, $j=1,\dots,n_{\mathrm{p}}$, $t=0,\dots, T-1 $, and all the constraints other than \eqref{eq:UC:demand} defining the set $\mathcal X$ where $x$ takes values. We can then validate our results using real-world hourly power demand data, obtained from the Spanish Transmission System Operator, Red El\'ectrica \cite{data2025redea}, covering the power demand from peninsular Spain between January 1st 2014 and December 31st 2024, for a total of $2870$ weekdays. 
Data are plotted in Figure~\ref{fig:Pd}.
Power is measured in GW, time in hours (h), cost in euros.

\begin{figure*}[t]
    \centering
    \begin{subfigure}[b]{0.24\linewidth}
        \includegraphics[width=\textwidth]{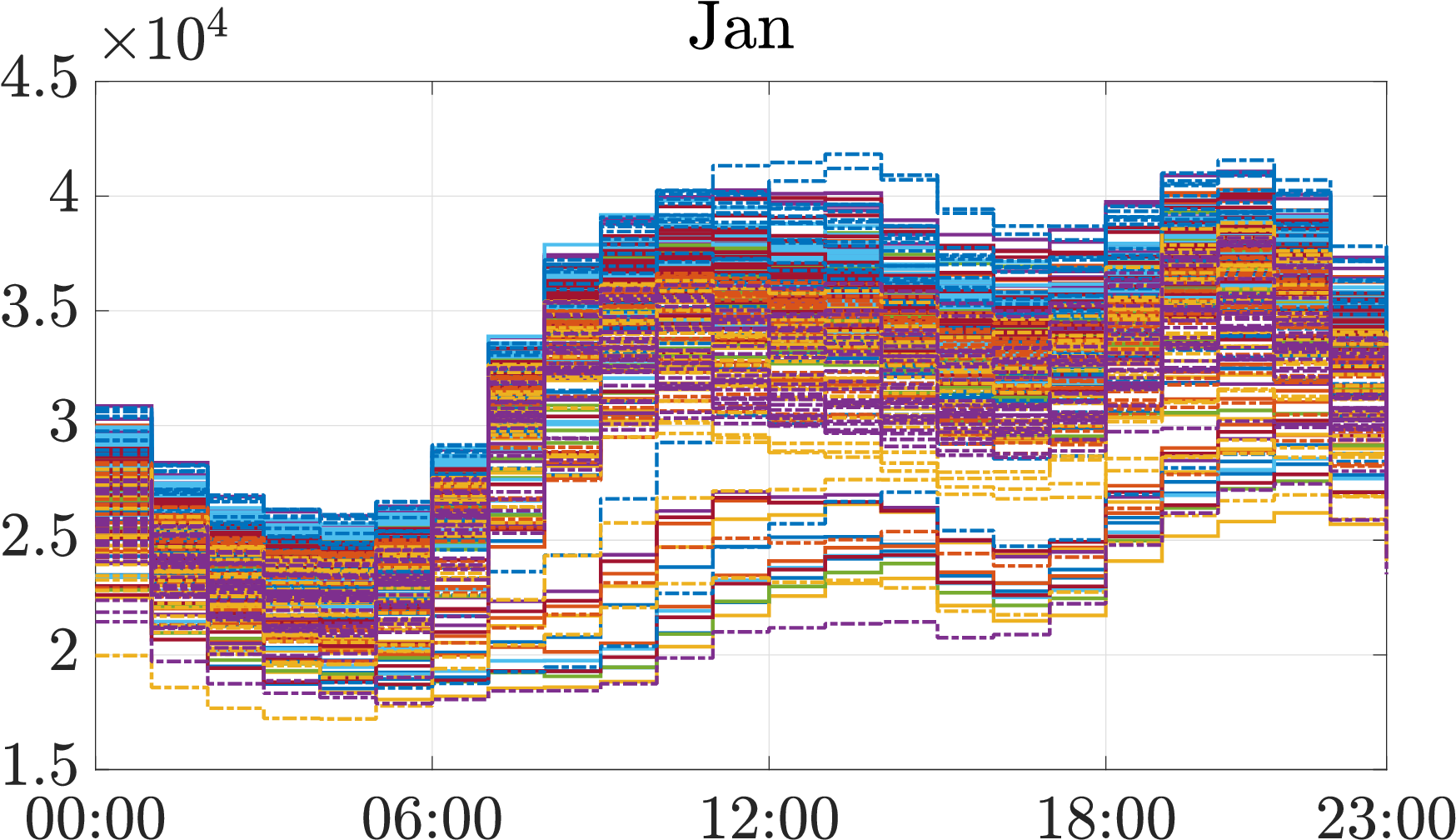}
    \end{subfigure}
    \begin{subfigure}[b]{0.24\linewidth}
        \includegraphics[width=\textwidth]{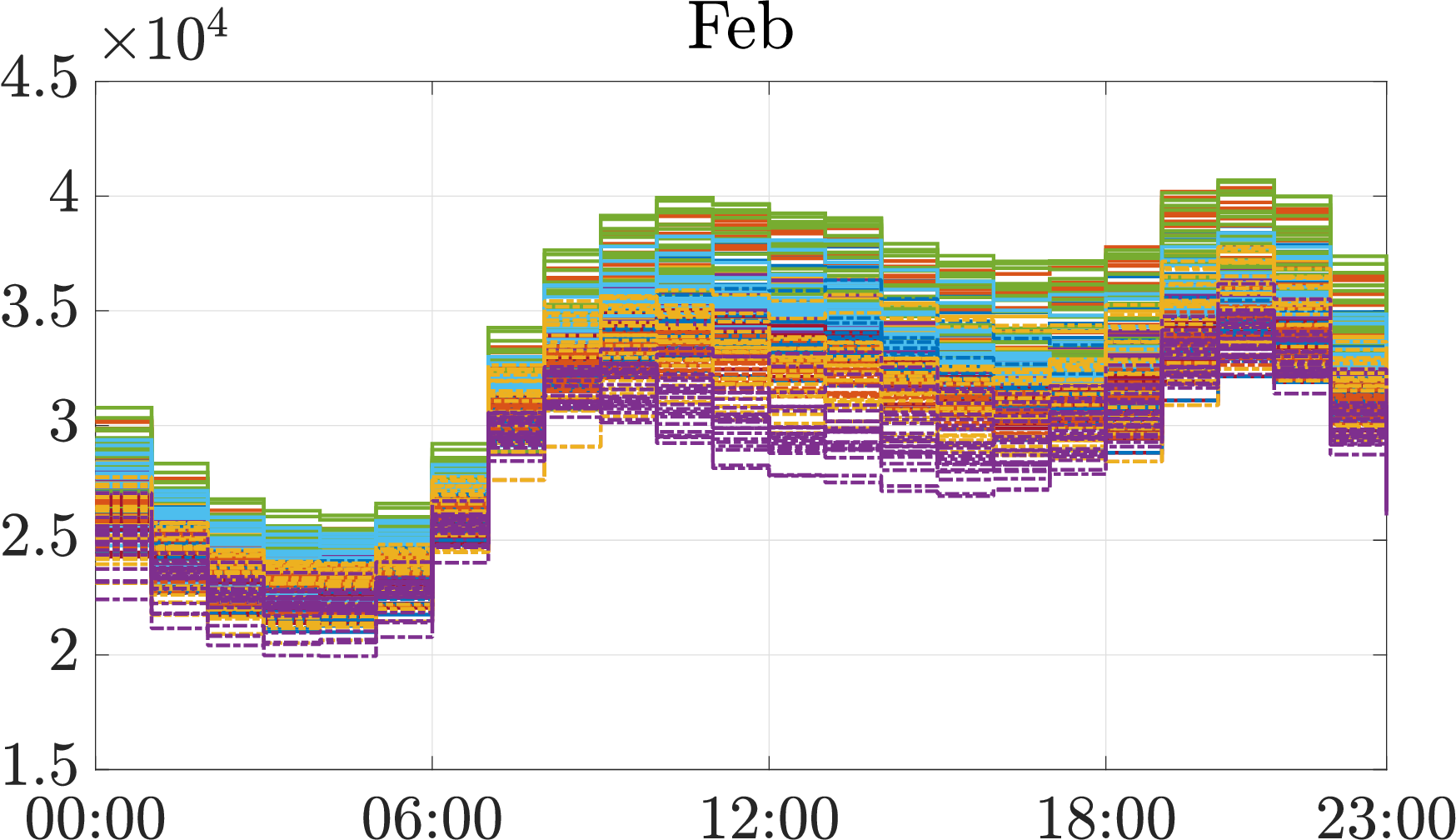}
    \end{subfigure}
    \begin{subfigure}[b]{0.24\linewidth}
        \includegraphics[width=\textwidth]{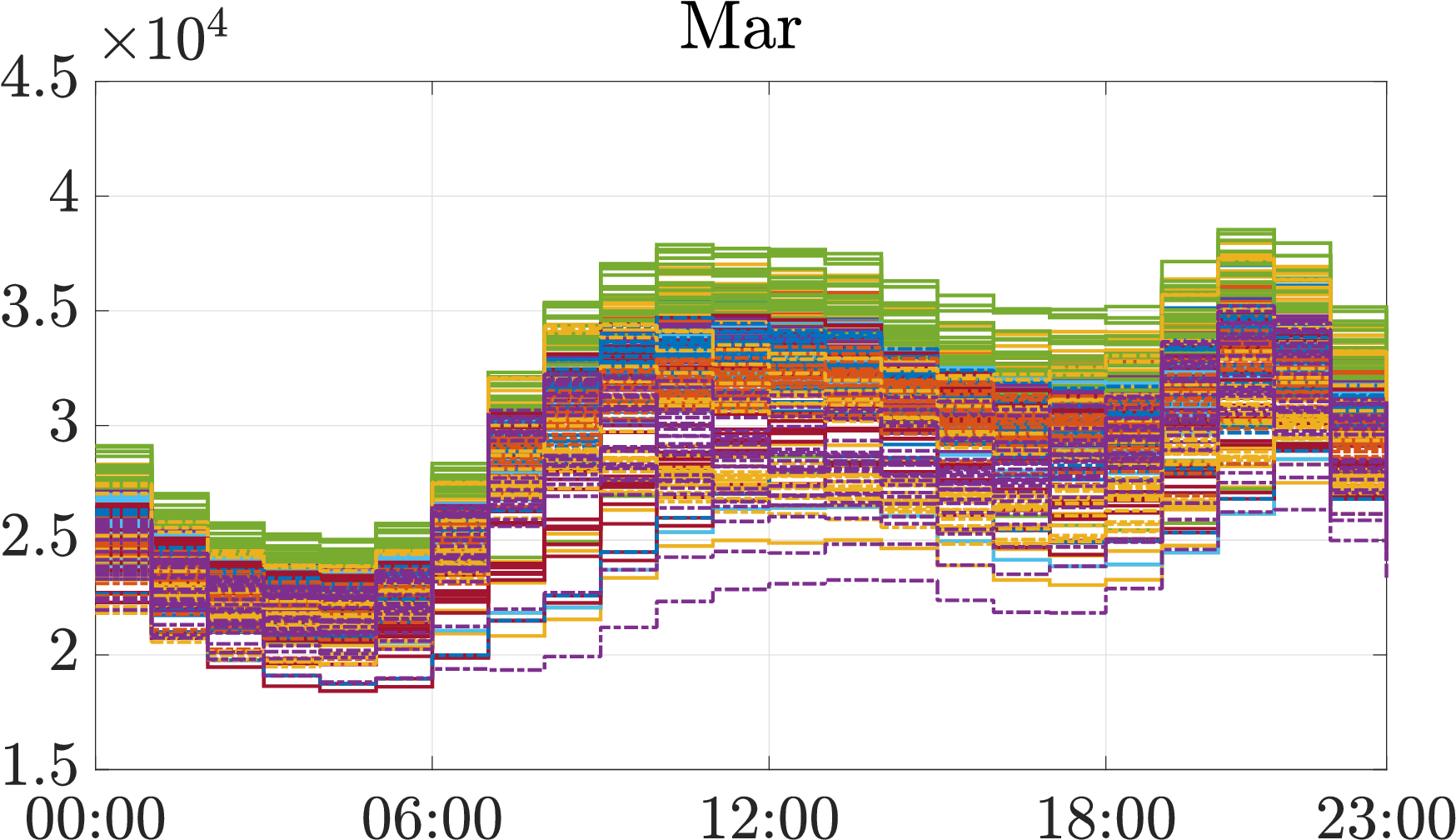}
    \end{subfigure}
    \begin{subfigure}[b]{0.24\linewidth}
        \includegraphics[width=\textwidth]{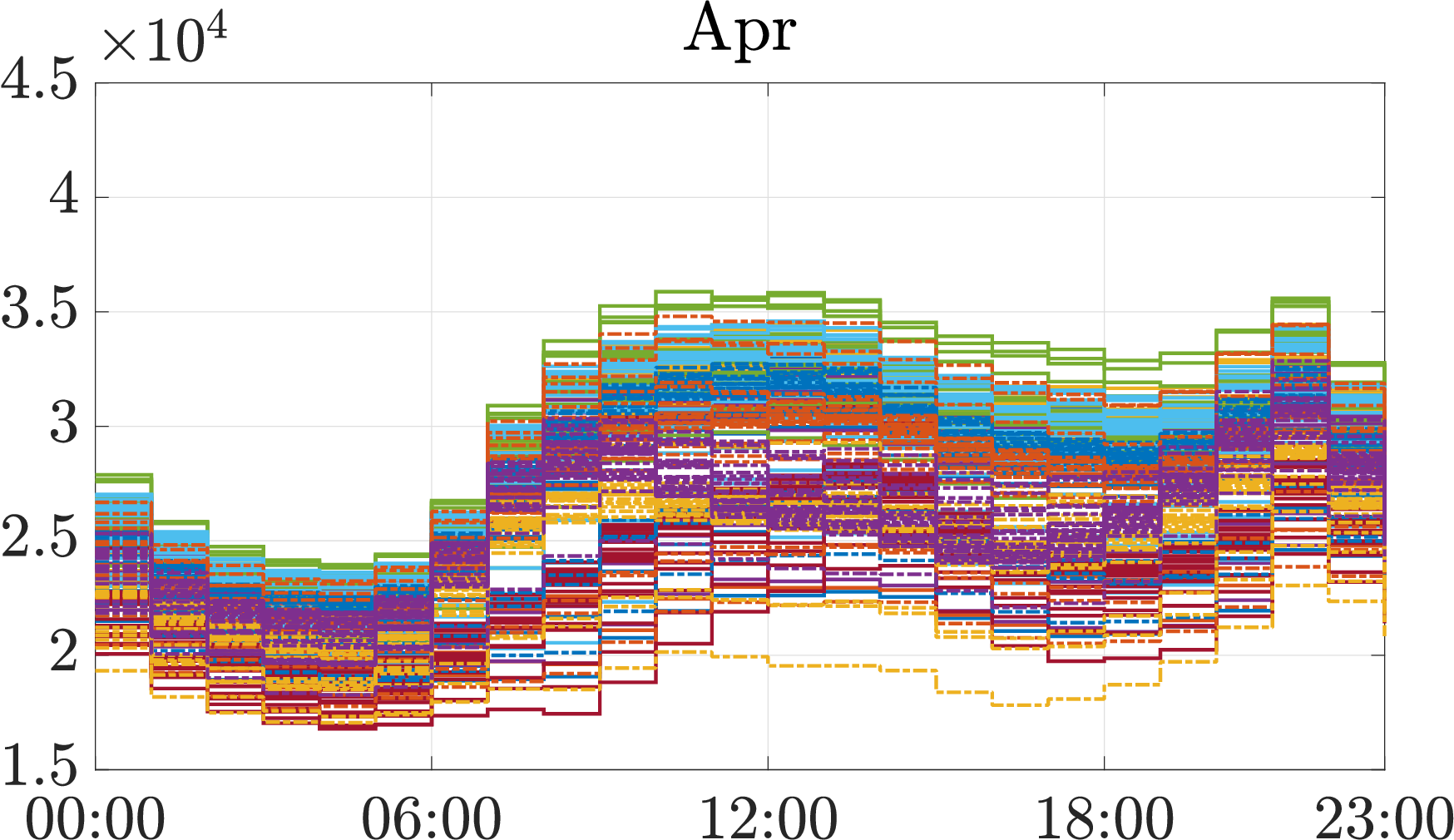}
    \end{subfigure}
        \\
        \vspace{.2cm}
    \begin{subfigure}[b]{0.24\linewidth}
        \includegraphics[width=\textwidth]{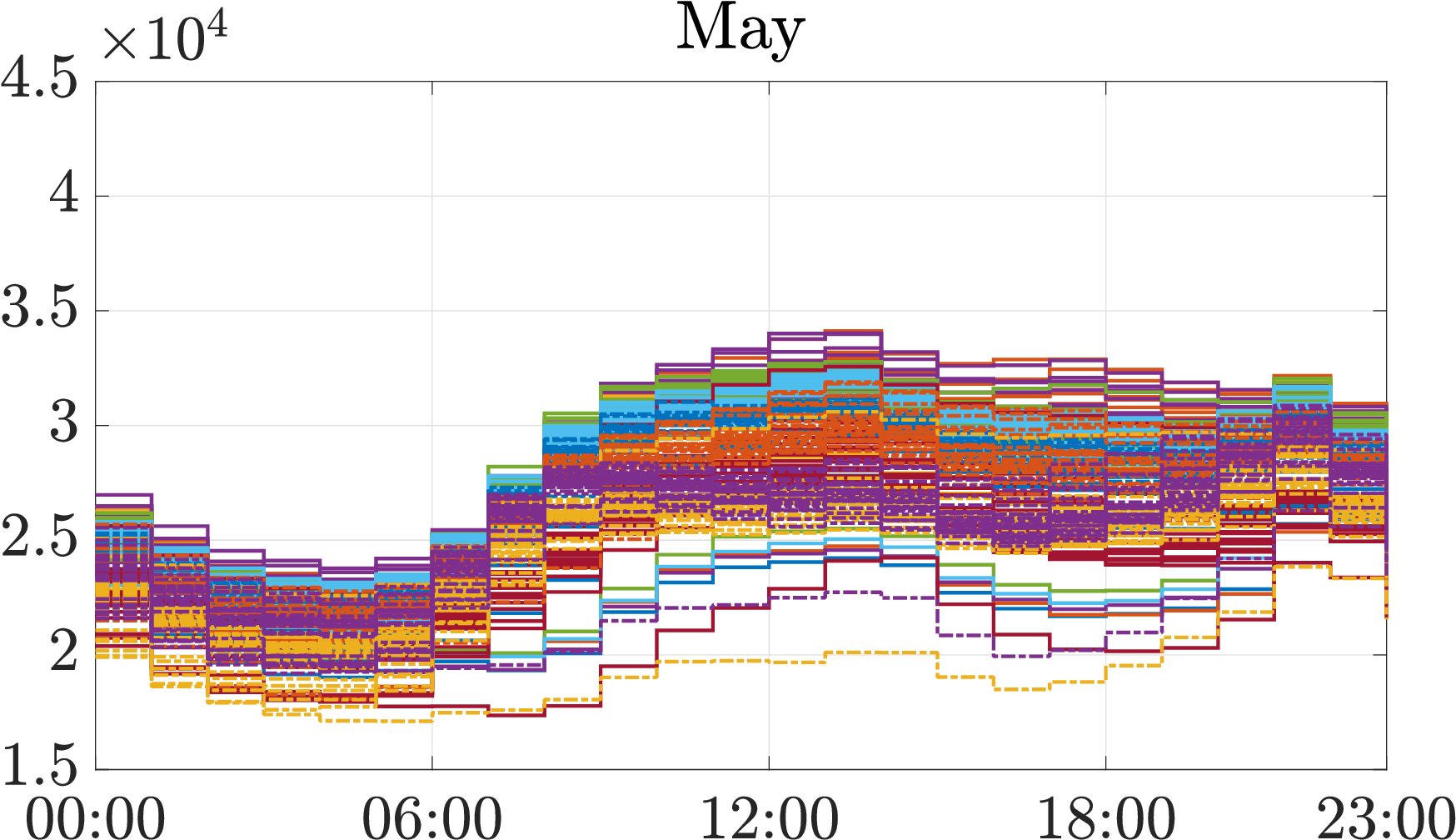}
    \end{subfigure}
    \begin{subfigure}[b]{0.24\linewidth}
        \includegraphics[width=\textwidth]{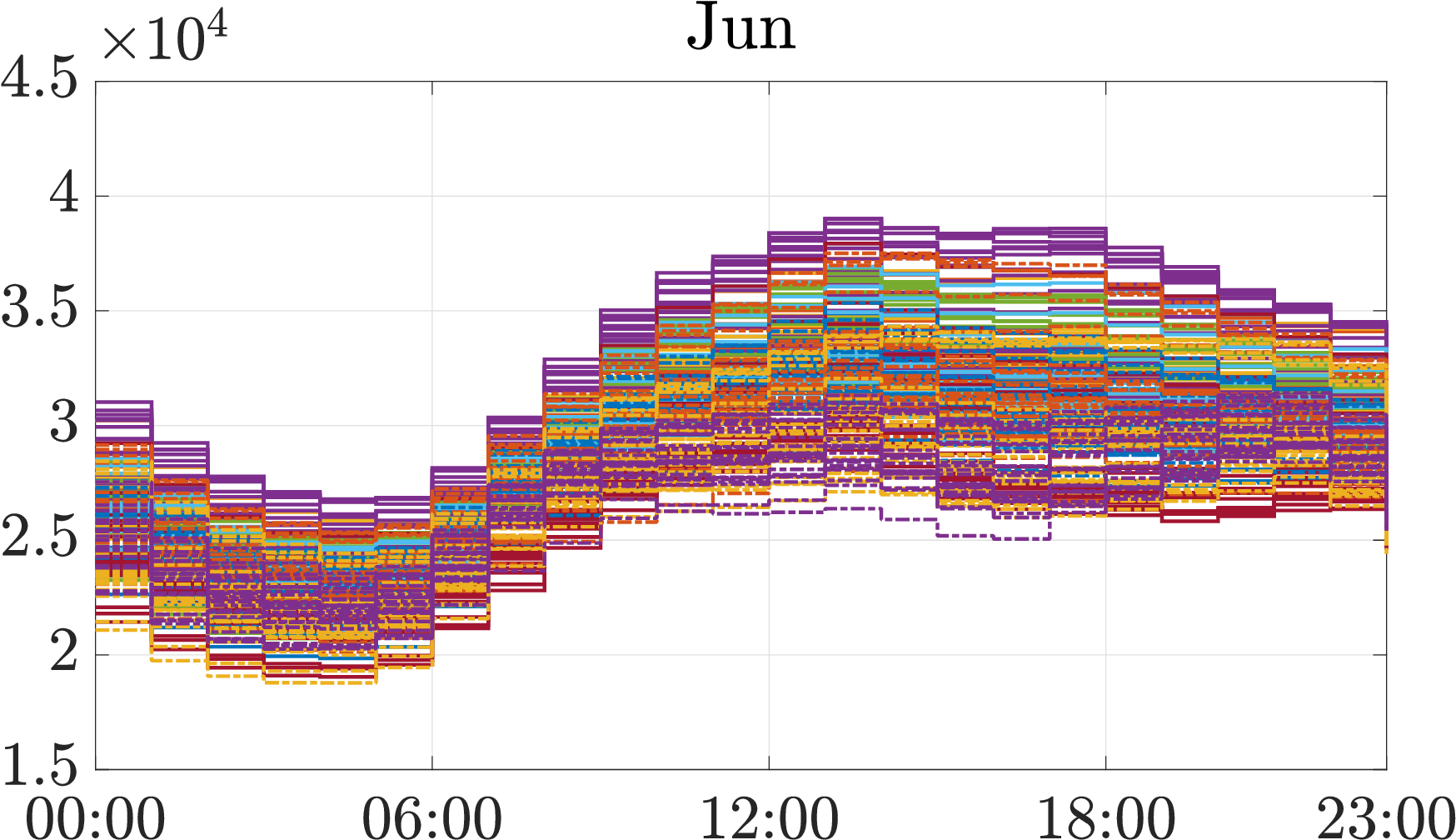}
    \end{subfigure}
    \begin{subfigure}[b]{0.24\linewidth}
        \includegraphics[width=\textwidth]{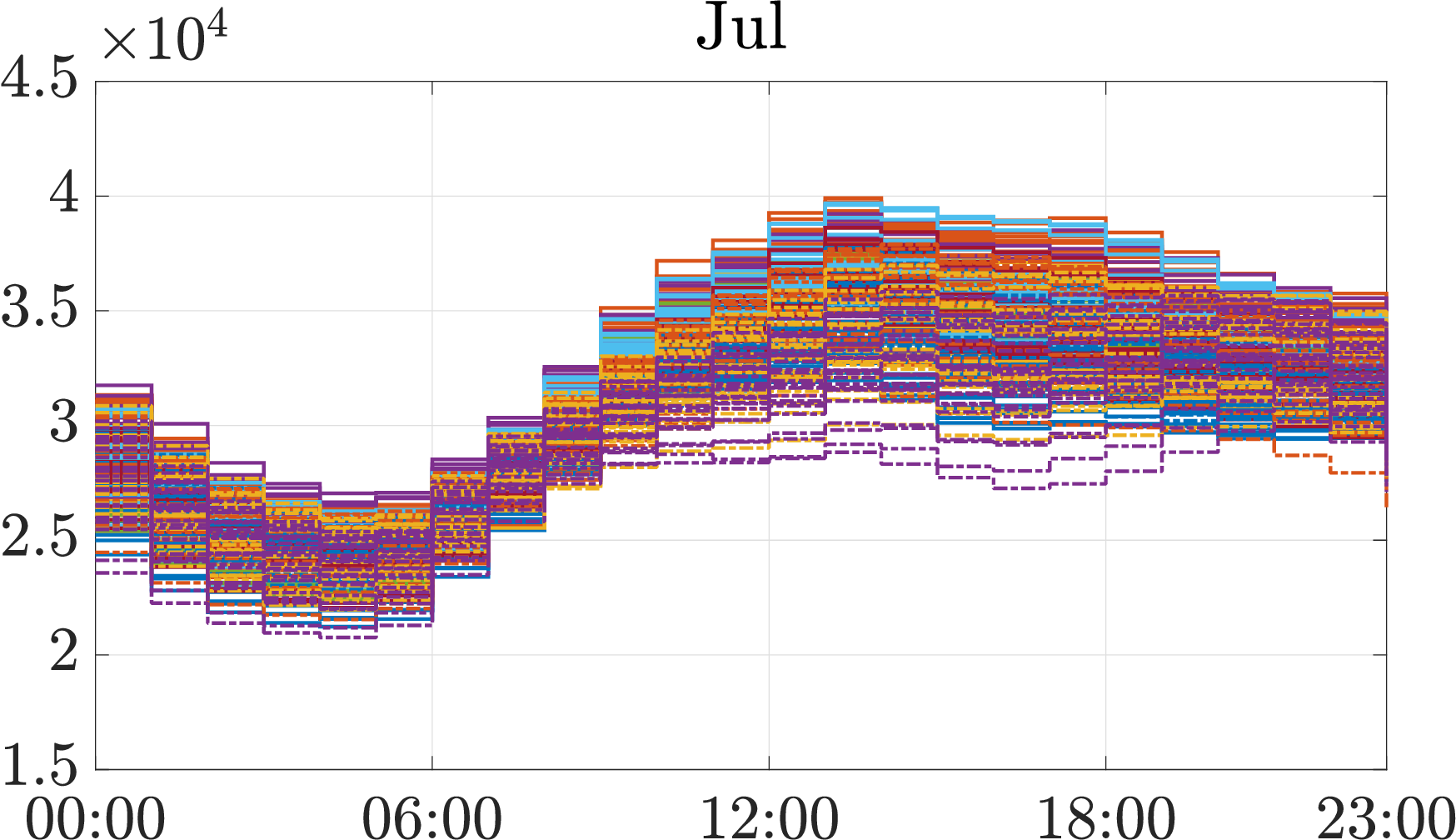}
    \end{subfigure}
    \begin{subfigure}[b]{0.24\linewidth}
        \includegraphics[width=\textwidth]{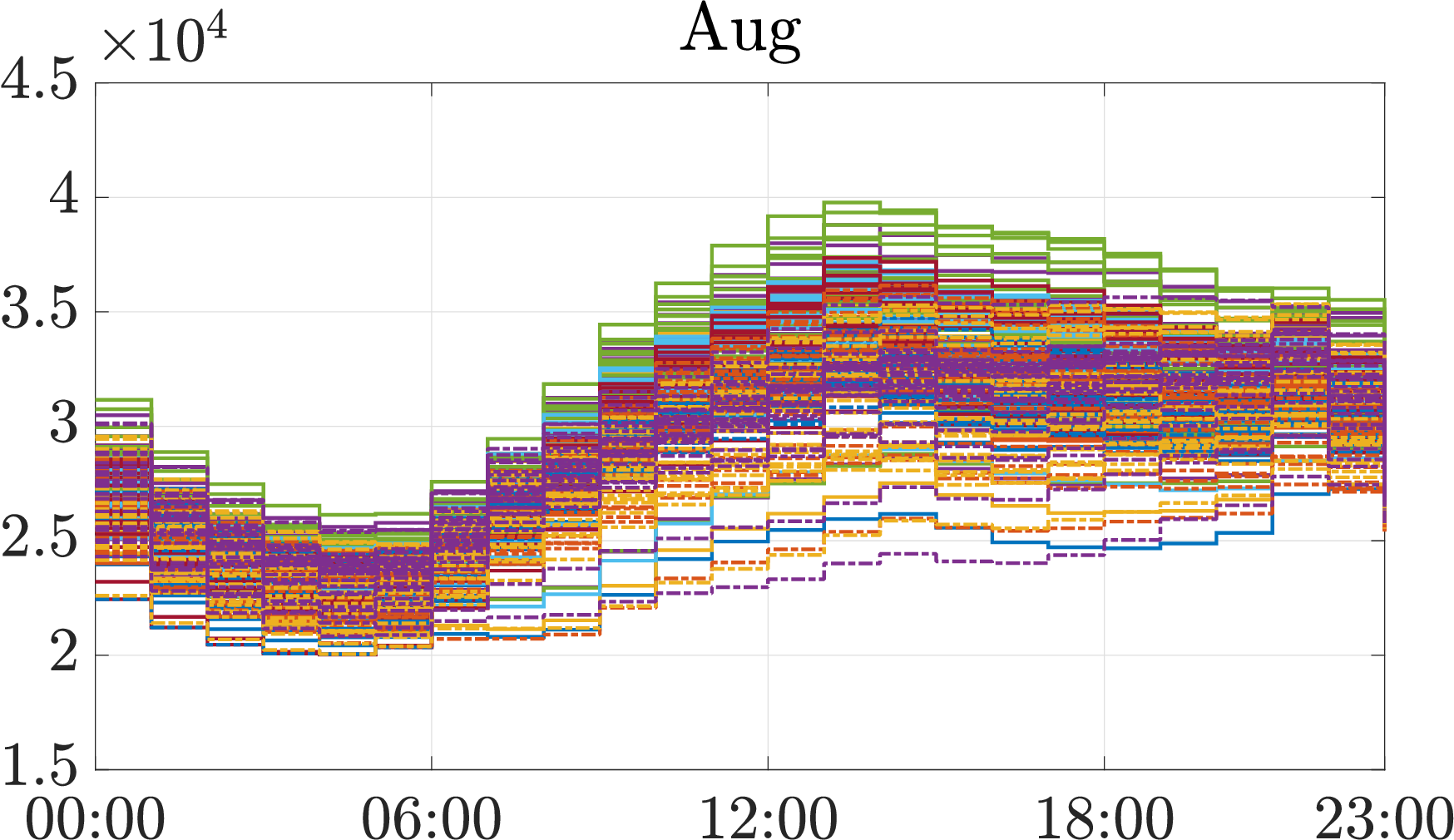}
    \end{subfigure}
        \\
        \vspace{.2cm}
    \begin{subfigure}[b]{0.24\linewidth}
        \includegraphics[width=\textwidth]{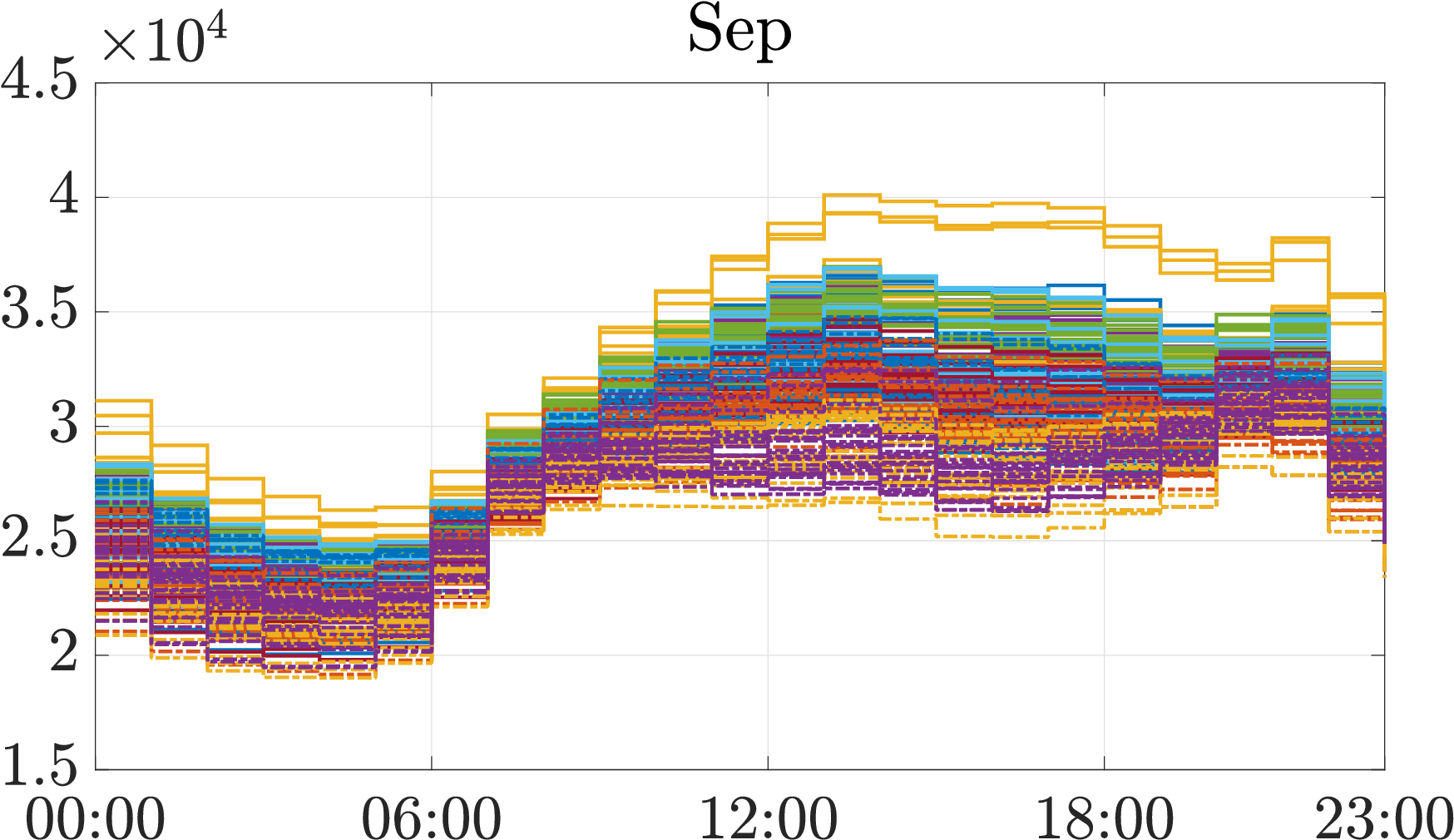}
    \end{subfigure}
    \begin{subfigure}[b]{0.24\linewidth}
        \includegraphics[width=\textwidth]{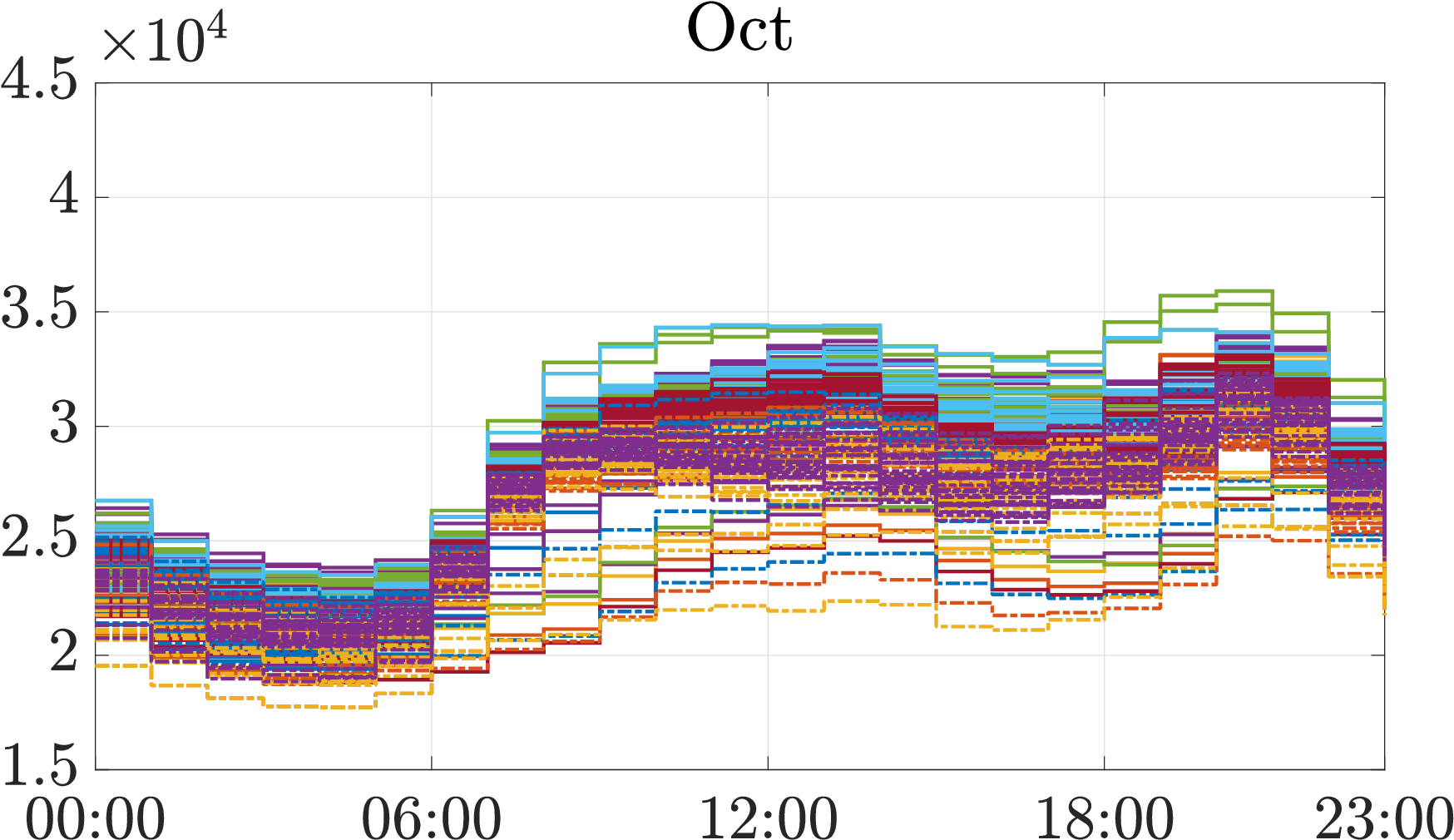}
    \end{subfigure}
    \begin{subfigure}[b]{0.24\linewidth}
        \includegraphics[width=\textwidth]{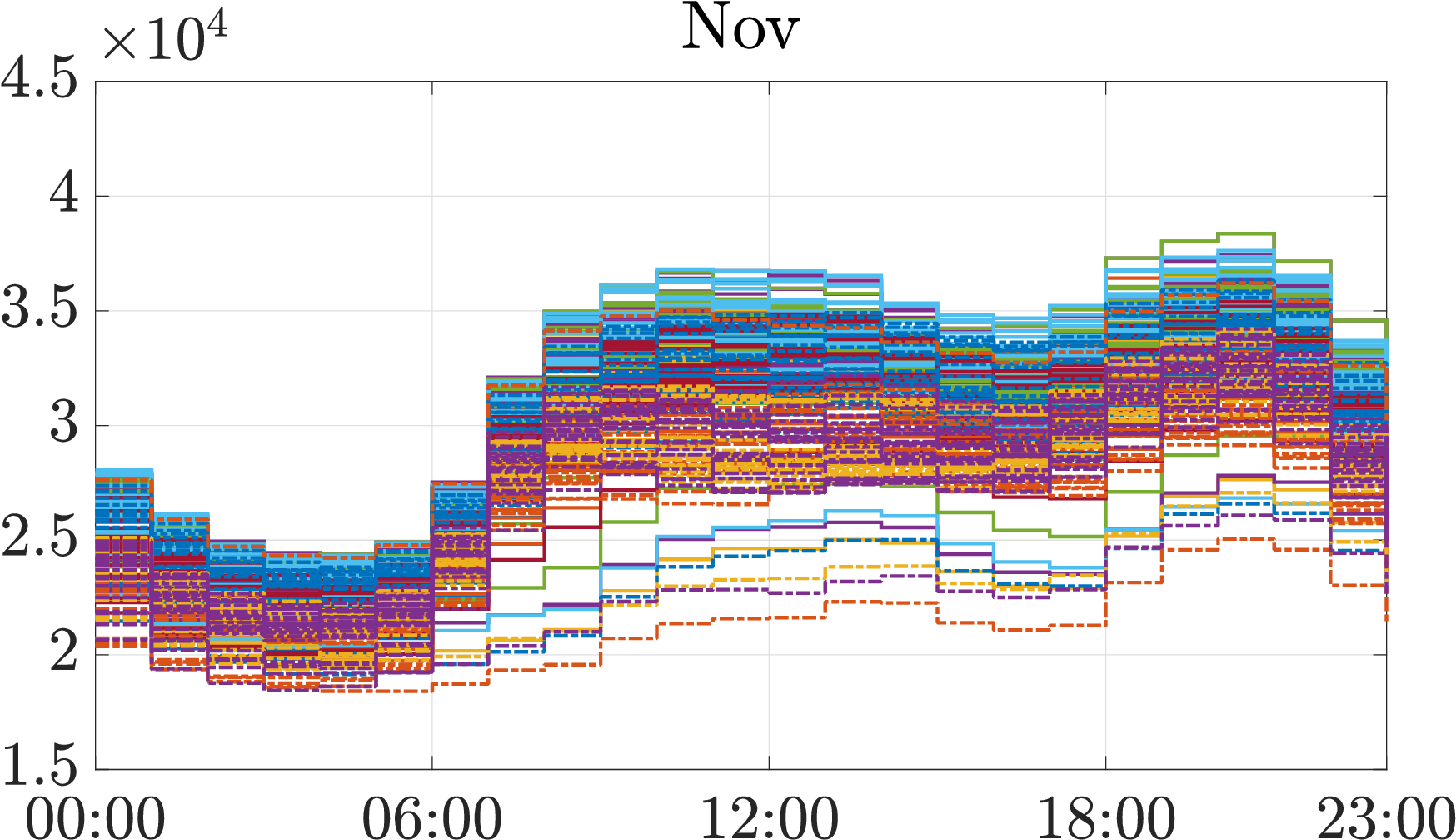}
    \end{subfigure}
    \begin{subfigure}[b]{0.24\linewidth}
        \includegraphics[width=\textwidth]{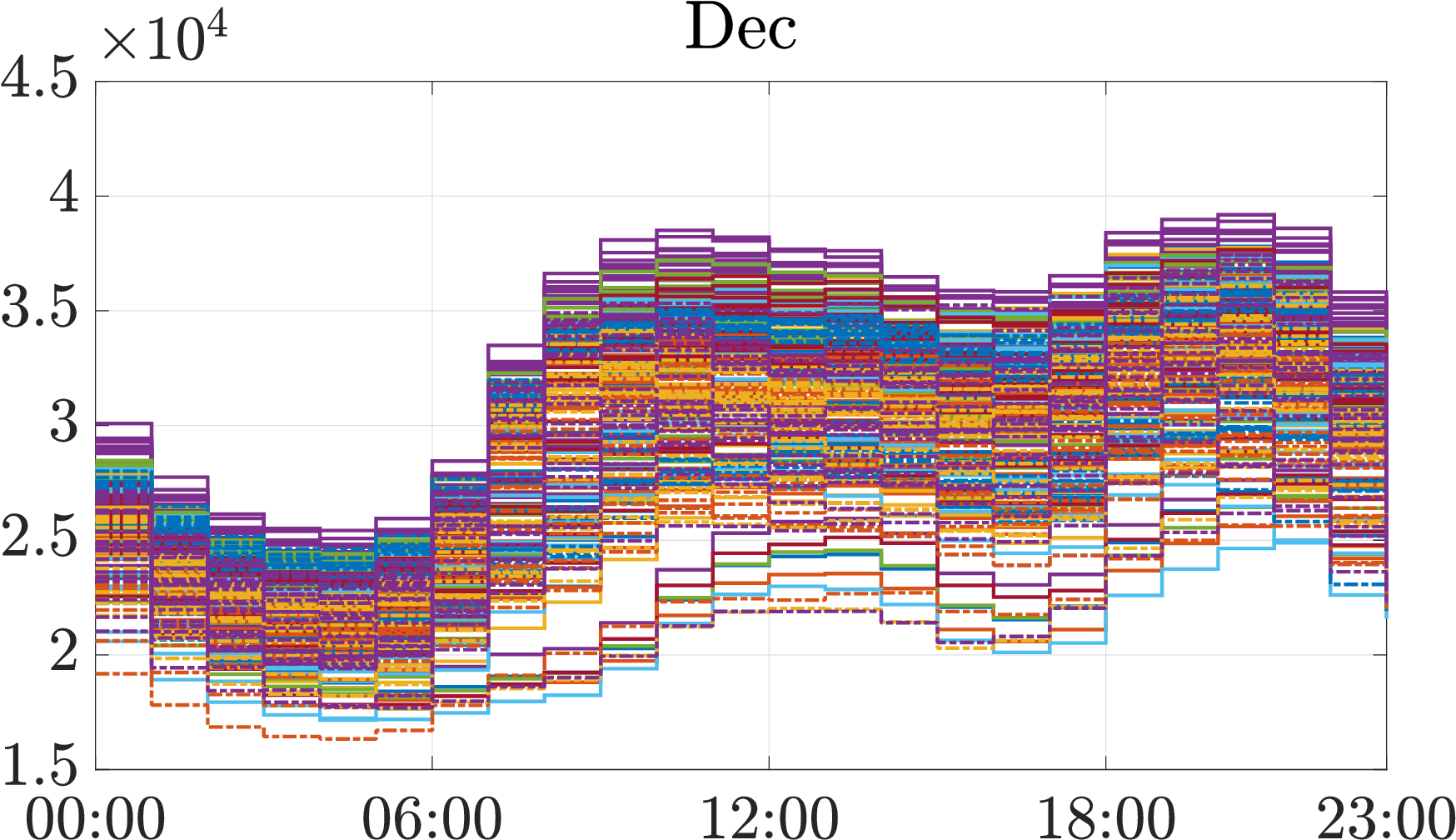}
    \end{subfigure}
        \\
        \vspace{.2cm}
    \begin{subfigure}[b]{\linewidth}
        \centering
        \includegraphics[width=0.75\textwidth]{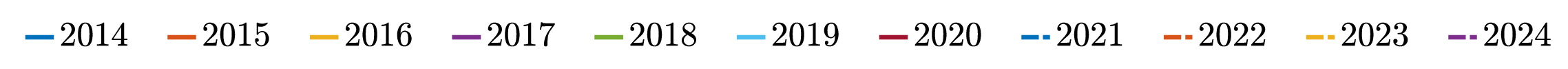}
    \end{subfigure}
    \caption{Weekday power demand of peninsular Spain between January 1st 2014 and December 31st 2024, obtained from Red El\'ectrica.
    }
    \label{fig:Pd}
\end{figure*}

\begin{table}[t]
    \centering
    \caption{Decision variables and parameters in the unit commitment scenario program \eqref{eq:UC} for $j=1,\dots,n_\mathrm{p}$, $t=0,\dots,T-1$, and $i=1,\dots,N$.}
    \setlength{\tabcolsep}{6pt}        
    \renewcommand{\arraystretch}{1.4}  
    \begin{tabular}{cc|l}
        \hline
        Decision variables 
        & {Domain}       & {Description} \\
        \hline
        $P_{j,t}$ & $\mathbb{R}_+$ & Power generated at time $t$  \\
        $Y_{j,t}$ & $\{0,1\}$      & ON/OFF status at time $t$ (auxiliary)\\
        $u_{j,t}$ & $\{0,1\}$      & Startup at time $t$ \\
        $d_{j,t}$ & $\{0,1\}$      & Shutdown at time $t$ \\
        $y_{j,z,t}$& $\{0,1\}$      & Production at time $t$ in 
        $[\underline{P}_{j,z},\overline{P}_{j,z}]$ \\[2pt]
        \hline
        GU parameters 
        & {Domain}       & {Description} \\
        \hline
        $(a_j,b_j,c_j)$ & $\mathbb{R}^3_+$ & Fuel cost coefficients\\
        $(c_j^u, c_j^d)$ & $\mathbb{R}^2_+$ & Startup and shutdown costs \\
        $(\underline{\Delta}_j^{\mathrm{p}},\overline{\Delta}_j^{\mathrm{p}})$ & $\mathbb{R}_+^2$ & Ramp limits on power \\
        $(T_j^{\mathrm{u}},T_j^{\mathrm{d}})$ &  $\mathbb{N}^2$ & Minimum on-time and off-time  \\
        $Z_j$ &  $\mathbb{N}$ & Number of operating zones\\
        $(\underline{P}_{j,z},\overline{P}_{j,z})$ & $\mathbb{R}^2_+$ & Extremes of the $z$-th operating region \\
        \hline
        Scenario parameters  
        & {Domain}       & {Description} \\
        \hline
        $P_{t,i}^{\mathrm{d}}$ & $\mathbb{R}_+$ & Power demand at $t$ in scenario $i$
        \\
    \end{tabular}
    \label{tab:UC:vars}
\end{table}

\subsection{Case study setup}

We next show the results obtained in a case study with $n_{\mathrm{p}} = 4$ GUs with parameters as in Table~\ref{tab:UC:params4}. 

\begin{table}[t]
    \setlength{\tabcolsep}{6pt}         
    \renewcommand{\arraystretch}{1.4}  
    \centering
    \caption{Case study parameters}
    \begin{tabular}{c|c|c|c|c}
        GU $j$ & 1 & 2 & 3 & 4 
        \\\hline
$a_j,b_j,c_j$ & $1,0.4,0.3$   & $0.3,2,0.2$ & $0.4,1,1$ & $10,0.1,0.1$ \\
$c_j^u,c_j^d$ & $0.9,0.4$ & $0.5,0.4$ & $0.2,0.3$ & $1, 0.8$\\
$\underline{\Delta}_j^p,\overline{\Delta}_j^p$ & $7,7$ & $2,0.2$ & $5,5$ & $1.5,1$ \\
$T_j^{\mathrm{u}},T_j^{\mathrm{d}}$ & $3,3$ & $2,1$ &$1,3$ & $1,4$  \\
$Z_j$ & 2 & 2 & 3 & 1\\   $\underline{P}_{j1},\overline{P}_{j1}$ &  $7,13.5$  & $1,3$ & $3,4$ & $1,13$ \\
$\underline{P}_{j2},\overline{P}_{j2}$ &  $13.8,14.5$  & $3.2,14.5$ & $8,9$  & - \\
$\underline{P}_{j3},\overline{P}_{j3}$ &  - & - & $13,14$ & - \\
    \end{tabular}
    \label{tab:UC:params4}
\end{table}

The unit commitment problem \eqref{eq:UC} has $d = 480$ decision variables, of which $96$ continuous and $384$ binary. 
The dimensionality of the program allows for its exact solution and, hence, for a comparison of the guarantees obtained with our method with those of the standard scenario theory revised in Section~\ref{sec:recap}, which applies to the optimal solution of the scenario program. 

In our analysis, we first consider the \textit{a posteriori} robustness certificate in Proposition~\ref{prop:a-posteriori-robustness} and compare it with the \textit{a posteriori} robustness guarantees obtained via Theorem~\ref{th:scenario} by determining a more accurate value of the complexity $s_{N}^{\star}$, which is a computationally expensive process. Therefore, the results are also compared in terms of computation time. As for the computation of $s_{N}^{\star}$, we adopt the greedy algorithm in \cite{campi2018general}, whereby the list of scenarios indexed by $\ell = 1,\dots, {N}$ is progressively reduced to a support list. Precisely, starting from $\ell=1$, $\delta_{\ell}$ is temporarily removed from the currently available scenarios. If the solution obtained by considering the constraints associated with the remaining scenarios coincides with the solution with all $N$ scenarios constraints in place, then $\delta_{\ell}$ is permanently discarded; otherwise, it is reinstated. Then, the procedure moves to the next $\ell$, and continues until $\ell = N$. The length of the resulting sub-list, which is irreducible by construction, but non necessarily minimal, is taken as value of the complexity $s_{N}^{\star}$.

We then consider data-set sizing for guaranteeing a certain desired bound $\bar \epsilon$ on the risk and compare the size of the data-sets obtained by the one-shot and incremental procedures in Section~\ref{sec:SampleSize} for the same confidence level $\beta$. Besides demonstrating the data saving, we also show that if we compute the optimal solutions using the data-set prescribed by the iterative procedure and the larger one of the one-shot procedure, then, the risk of the former estimated \textit{a posteriori} through Proposition~\ref{prop:a-posteriori-robustness} is closer to $\bar \epsilon$ than that of the latter. Moreover, the risk of the corresponding scenario solutions is shown to be slightly closer to $\bar\epsilon$ when the incremental procedure is used. 

All tests are performed on a laptop with an Intel\textsuperscript{\textregistered} Core\texttrademark~Ultra 7 155H, 1400Mhz, with 16 cores and 22 logical processors, and 32.0 GB of RAM. 

\subsection{Robustness assessment}\label{subsec:UCP-robustness}

Our experimental analysis was run by solving problem \eqref{eq:UC} 12 times, using half of the available weekday power demand data of 3 consecutive months centered at month $k=1,\dots,12$, from January ($k=1$) to December ($k=12$), to better meet the identically distributed assumption required by the scenario theory. The remaining half data within the same months were used for Monte-Carlo evaluations of the actual risk. 

Note that we shall introduce the superscript $k$ in all the quantities involved to recall what the central month is, but we shall use the same subscript $N$ (without any $k$) to denote that quantities depend on the number of considered scenarios, even if this number slightly changes depending on the considered 3-month period. 

In Figure~\ref{fig:UC:complexity4} we represent the estimated value of the complexity $s_{N}^{\star,k}$, together with the estimate $\varsigma_{N}^k$ of the complexity of problem \eqref{eq:SP:c} using the same data, for each of the 12 optimal solutions $x_N^{\star,k}$, with $k=1,\dots,12$. 
\begin{figure}
    \centering
    \includegraphics[width=0.99\linewidth]{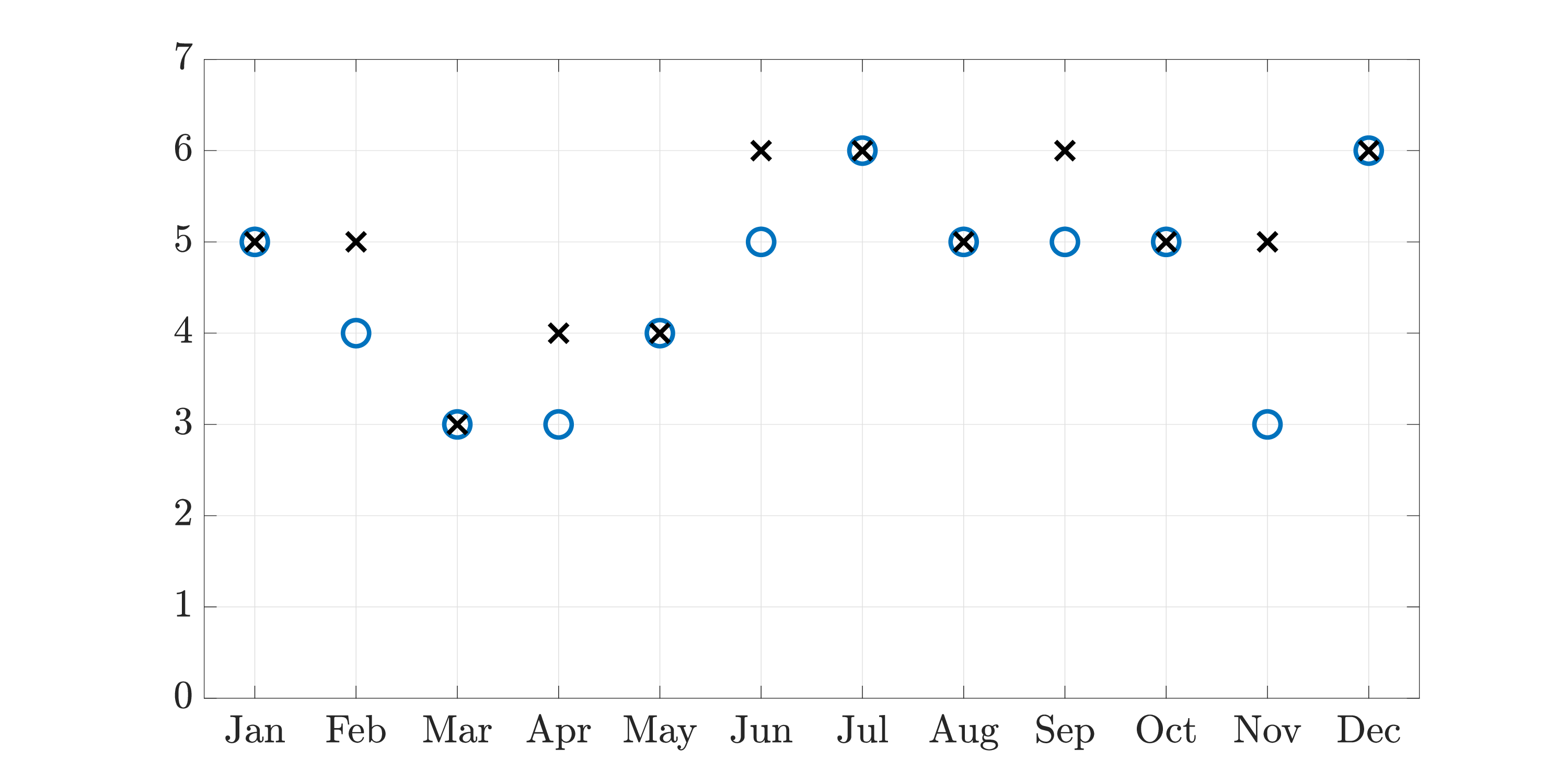}
    \caption{Comparison of $s_N^{\star,k}$ (circles) with $\varsigma_N^k$ (crosses), for $k = 1,\dots,12$, with the central month indicated on the horizontal axis.}
    \label{fig:UC:complexity4}
\end{figure}
As expected, $s_N^{\star,k}\le \varsigma_N^k$, $k=1,\dots,12$, although the difference is minor. As shown in Figure~\ref{fig:solution} with reference to $k=6$ (central month June), when $s_N^{\star,k}< \varsigma_N^k$,  there are some time slots during the one-day horizon  where the power provided by the committed GUs is larger than the maximal power demand over the scenarios, which means that the operating constraints are dominating over the scenario constraints that are not active (cf. discussion after Theorem~\ref{th:comp} and Figure~\ref{fig:gVxi}).

\begin{figure}
    \centering
    \includegraphics[width=0.99\linewidth]{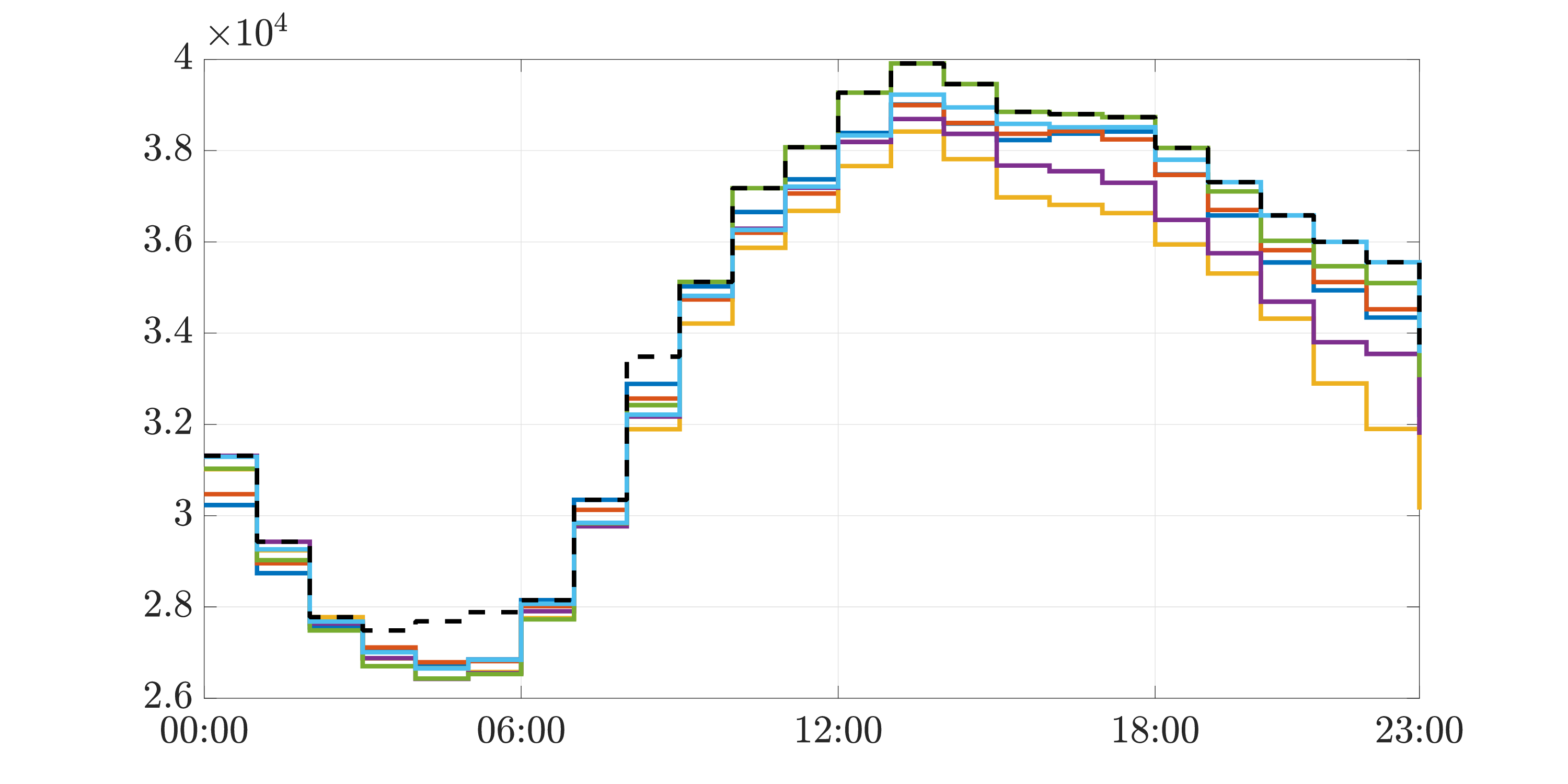}
    \caption{Comparison of $\sum_{j = 1}^{n_\mathrm{p}} P_{j,t}^{\star}$ (black dashed line) with $P_{i,t}^\mathrm{d}$, $i \in \mathcal I_N$ (solid colored lines), for $k=6$.}
    \label{fig:solution}
\end{figure}

Figure~\ref{fig:UC:bound4} reports the empirical risk, and the values of the \textit{a priori} bound on the risk in \eqref{eq:nc:aPri} and the \textit{a posteriori} bound computed according to \eqref{eq:a-posteriori-bound} and \eqref{eq:eps-1-beta}, respectively using $\varsigma_N^k$ and $s_N^{\star,k}$. All bounds were assessed for the same confidence parameter $\beta=10^{-6}$.  Not surprisingly, the \textit{a priori} bound is the most conservative and the \textit{a posteriori} bound relying on the complexity $s_N^{\star,k}$ is better than the one relying on $\varsigma_N^k$, but only for those $k$ such that $s_N^{\star,k}< \varsigma_N^k$ (cf. Figure~\ref{fig:UC:complexity4}). All bounds are above the actual risk as predicted by the theory.

\begin{figure}
    \centering
    \includegraphics[width=0.99\linewidth]{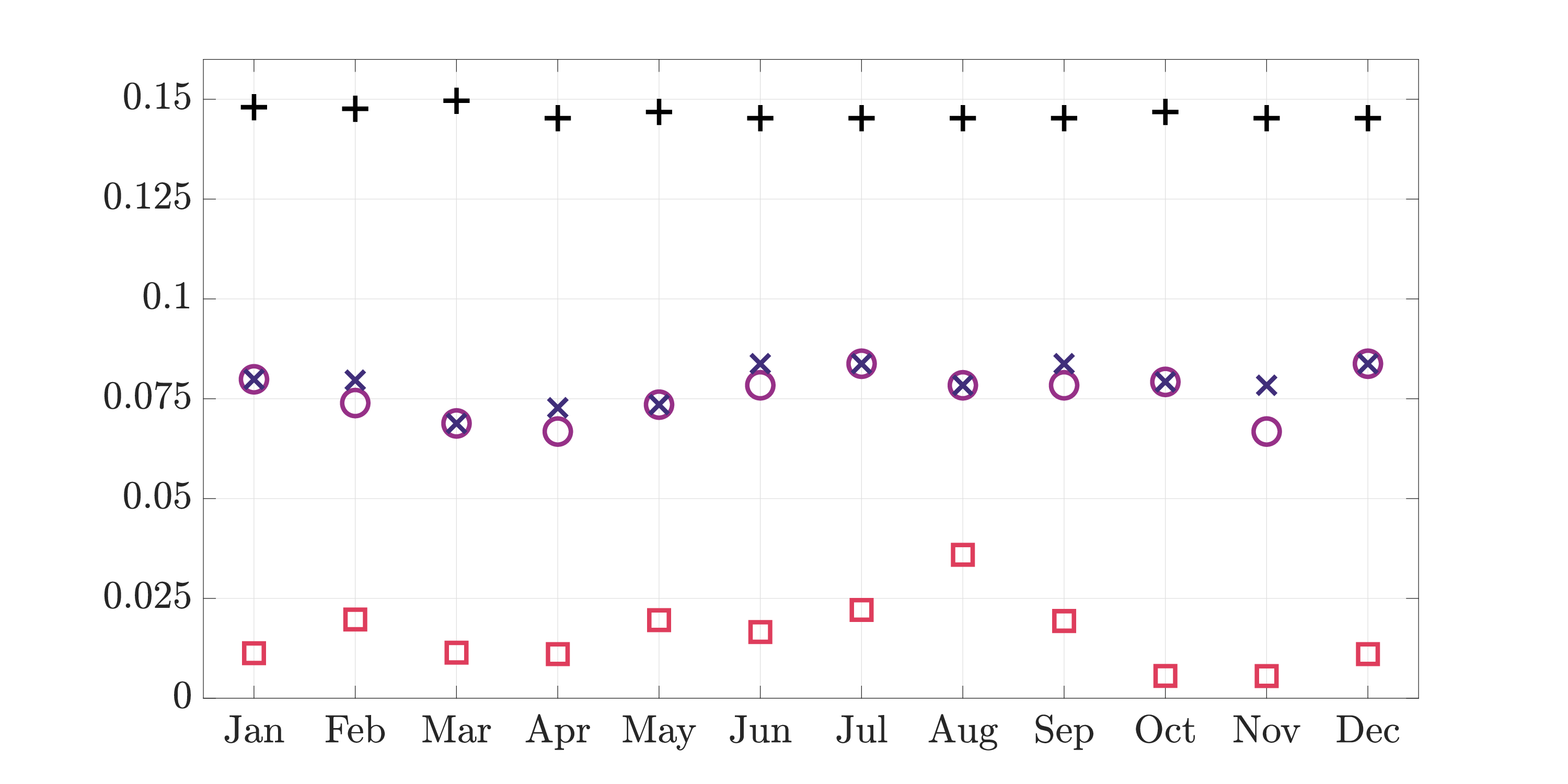}
    \caption{Comparison of the empirical risk of $x_N^{\star,k}$ (squares) with its \textit{a posteriori} bounds obtained via \eqref{eq:eps-1-beta}  (circles) and \eqref{eq:a-posteriori-bound} (crosses), and the \textit{a priori} bound in \eqref{eq:nc:aPri} (pluses), for $k = 1,\dots,12$, with the central month indicated on the horizontal axis.
    }
    \label{fig:UC:bound4}
\end{figure}

The average time for determining the \textit{a posteriori} bounds in Figure~\ref{fig:UC:bound4} was $4\,\text{h}$ and $33\,\text{min}$ for bound  \eqref{eq:eps-1-beta} and $0.5\, \text{ms}$ for bound \eqref{eq:a-posteriori-bound}. The motivation for this significant difference is that while computing $\varsigma_N^k$ is quite straightforward, computing $s_N^{\star,k}$ is time consuming, even with the adopted greedy procedure in \cite{campi2018general}, since one has to repeatedly solve the mixed integer unit commitment problem \eqref{eq:UC}. 

\subsection{Experiment sizing}

We also computed the size of the data-set necessary to achieve a prescribed bound  $\bar\epsilon=0.1$ on the risk, for ${\beta=10^{-6}}$. 
In this case, we compared the value of $N = 533 $ obtained with the one-shot method in Theorem~\ref{th:1shot} with the value $N_{j^\star}$ obtained by running the incremental Algorithm~\ref{alg:Nj} on scenarios extracted at random from half of the power demand data of the 5 months centered in July, amounting to a total of $601$ daily power demand profiles. The other half of the data within the same months is used for risk estimation.
Given that the outcome of the incremental approach depends on the considered scenarios, we performed 100 runs and built the histogram of Figure~\ref{fig:UC:iter4:hist}. In all runs, we found a value of $N_{j^\star}$ smaller than $N=533$ of the one-shot method, with a worst-case of $376$.

In every run, besides the optimal scenario solution obtained using the $N_{j^\star}$ scenarios of the data-set constructed with Algorithm~\ref{alg:Nj}, we also determined the optimal solution of the scenario program associated with that data-set enlarged with additional $533-N_{j^\star}$ profiles so as to reach the size $N=533$ prescribed by the one-shot method.  
Figure~\ref{fig:UC:iter4} shows the results of the 100 runs by comparing the \textit{a posteriori} bound \eqref{eq:a-posteriori-bound} on the risk for the two optimal solutions (panel a) and their corresponding Monte-Carlo estimated actual risk (panel b). Note that the \textit{a posteriori} bound is closer to the desired $\bar \epsilon$ when  $N_{j^\star}$ scenarios determined by the incremental method are used. The risk is nevertheless slightly overestimated since the gap with the actual risk remains a bit large for both bounds; just slightly smaller for the optimal solution of the incremental method, which, in any case, has the advantage of using a lower number of scenarios.

\begin{figure}
    \centering
    \includegraphics[width=0.99\linewidth]{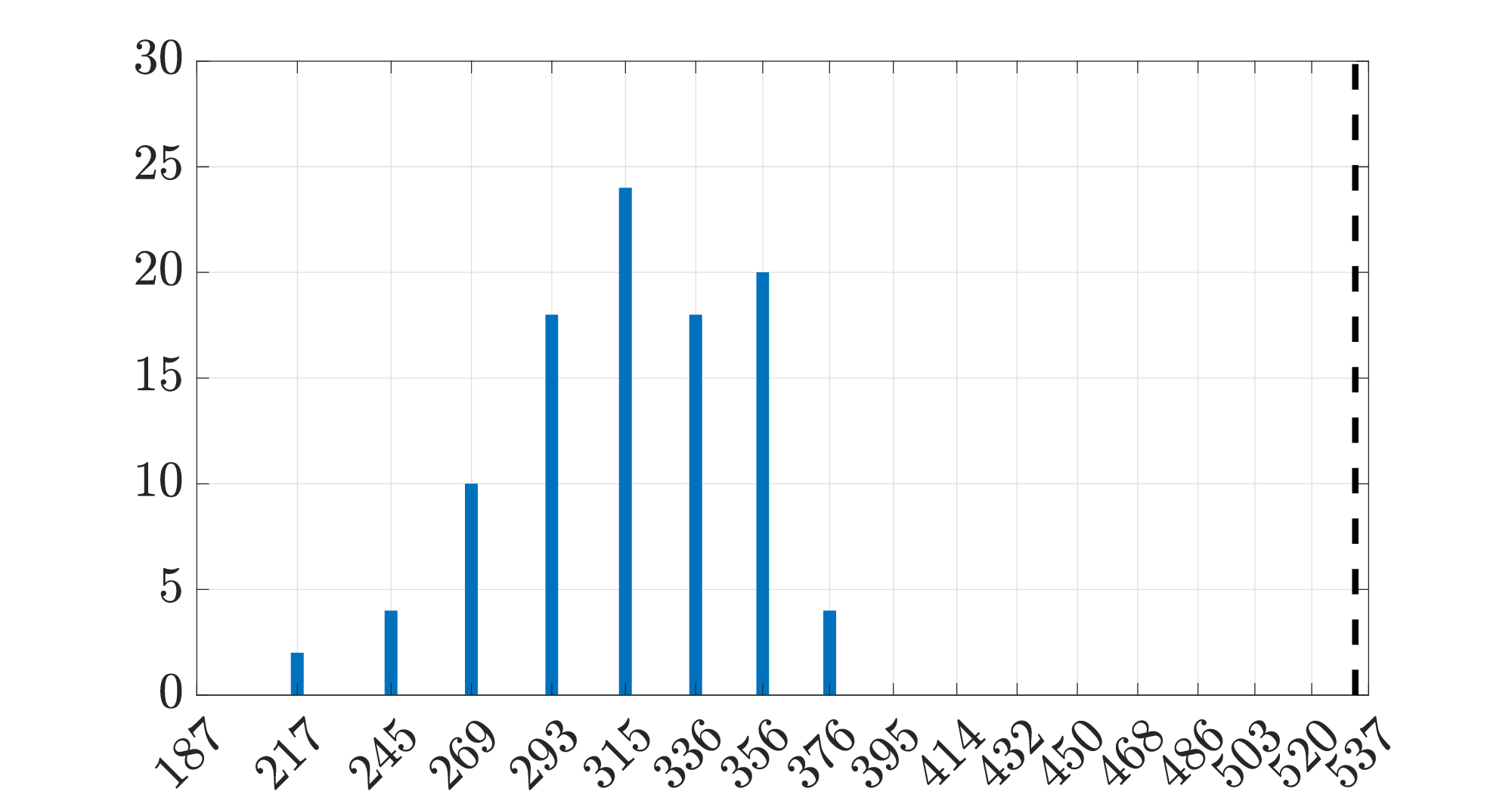}
    \caption{Number of occurrences of $N_{j^\star}$  obtained over $100$ runs of the incremental approach in Algorithm~\ref{alg:Nj} compared to the value $N$ obtained with the one-shot method in \eqref{eq:oneshot} (vertical dashed line). 
    }
    \label{fig:UC:iter4:hist}
\end{figure}

\begin{figure}[t]
    \centering
    \begin{subfigure}[b]{0.8\linewidth}
        \includegraphics[width=\textwidth]{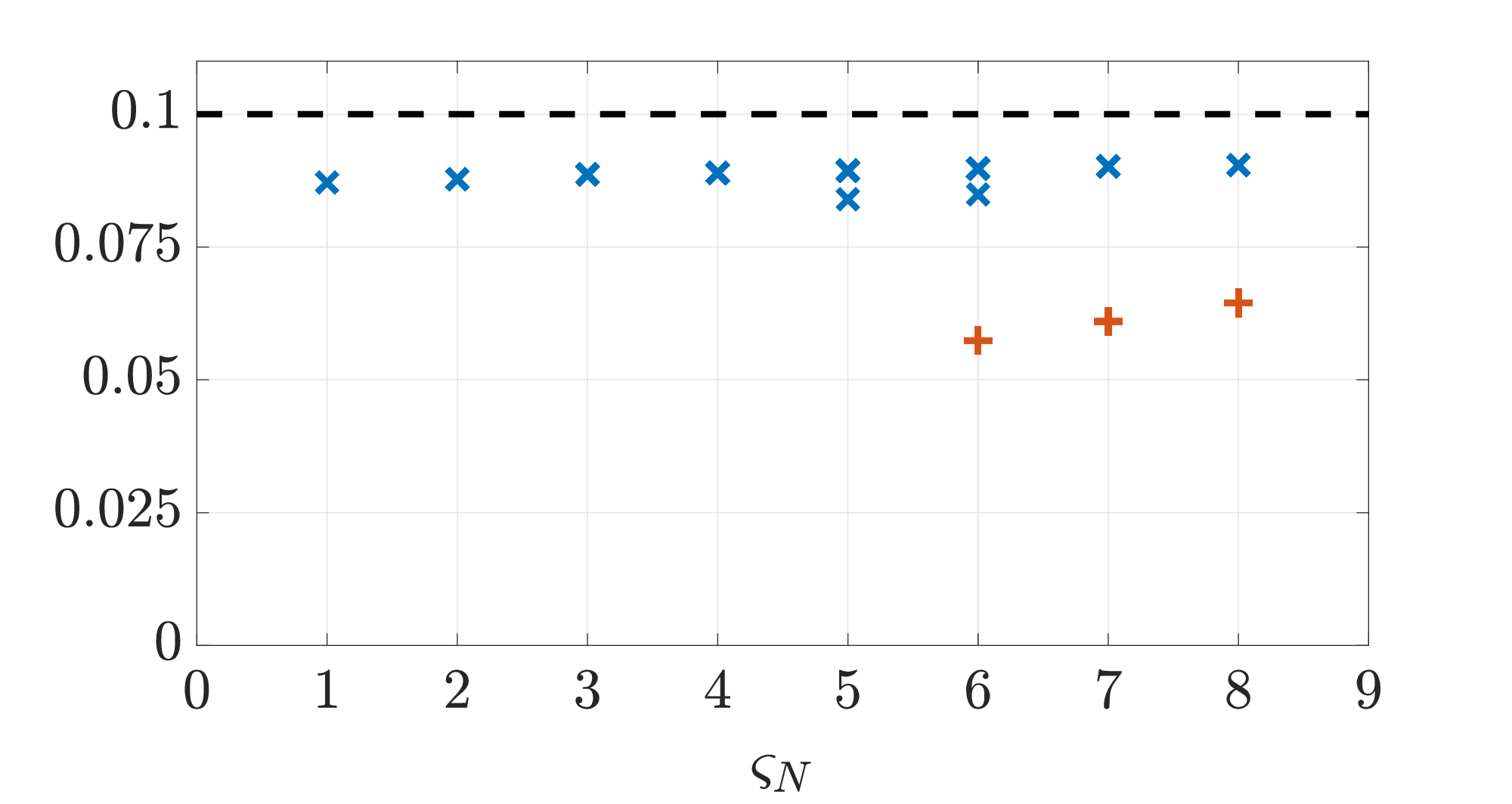}
        \caption{A posteriori bound on the risk.}
        \label{fig:UC:iter4:1S}
    \end{subfigure}\\
    \begin{subfigure}[b]{0.8\linewidth}
        \includegraphics[width=\textwidth]{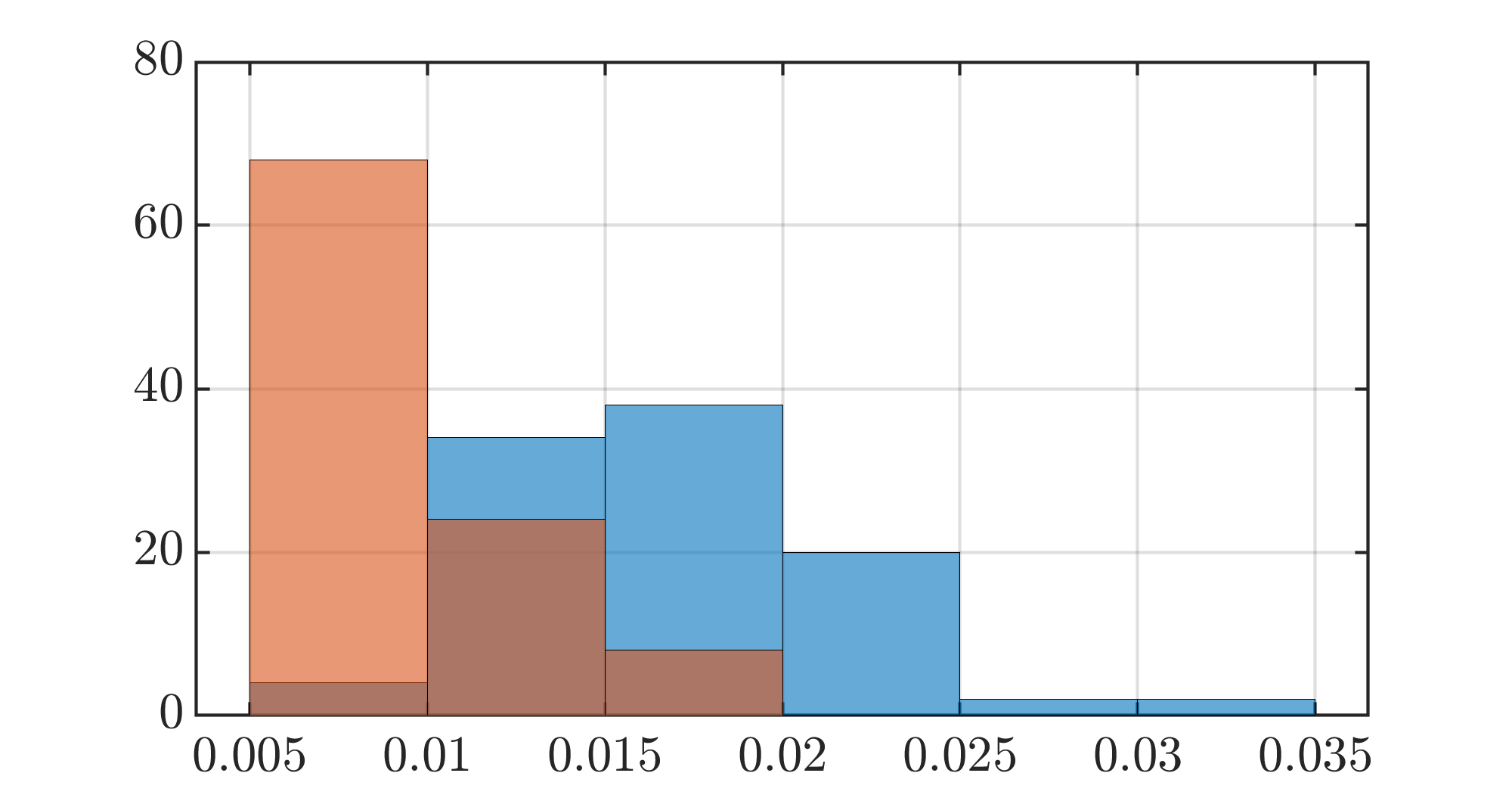}
        \caption{Empirical risk.}
        \label{fig:UC:iter4:iter}
    \end{subfigure} 
    \caption{Risk assessment over $100$ runs. Figure (a): A posteriori bound \eqref{eq:a-posteriori-bound} on the risk as a function of the complexity $\varsigma_N$ obtained when using the one-shot method (orange pluses) and the incremental method (blue crosses) for experiment sizing, with the desired $\bar \epsilon =0.1$ plotted with a black dashed line. Figure (b): Histograms of the empirical risk of the scenario solutions obtained with the data-set of the one-shot method (in orange) and the incremental method (in blue).  }
    \label{fig:UC:iter4}
\end{figure}

\section{Conclusion}\label{sec:Concl}
In this paper, we consider data-based optimization problems where uncertainty enters the constraints in a additively  way, so that it  can be separated from the decision variables.
We provide results on probabilistic \textit{a priori} and \textit{a posteriori} robustness certificates and experimental data sizing for achieving a certain probabilistic robustness level. Unlike standard applications of the scenario theory, our guarantees hold for a possibly sub-optimal solution and can be computed with negligible effort before solving the optimization problem, which is convenient in a non-convex context where computing an optimal solution is computationally demanding. Admittedly, this comes at the price of more conservative certificates than those potentially achievable by the scenario approach, which, however, apply to the optimal solution only and require computing the complexity, involving to repeatedly solve a non-convex program. We formulated the unit commitment problem in electric power system as a mixed integer quadratic program and analyzed our results comparatively to those of standard scenario theory using a low dimensional instance of the problem so as to be able to compute the optimal solution. The analysis performed on real power demand data shows that we experience a significant gain in computational effort at the price of a limited increase in conservativeness.

\bibliographystyle{ieeetr}
\bibliography{ref}

\end{document}